\documentclass[12pt]{amsart}
\usepackage{amssymb,amsmath,amsfonts,latexsym}
\usepackage{xcolor}

\newtheorem{thm}{Theorem}[section]
\newtheorem{prop}[thm]{Proposition}
\newtheorem{lem}[thm]{Lemma}
\newtheorem{cor}[thm]{Corollary}

\newtheorem*{theorem*}{Theorem}
\newtheorem*{statement*}{Statement}

\theoremstyle{definition}
\newtheorem{definition}[thm]{Definition}
\newtheorem{example}[thm]{Example}

\theoremstyle{remark}
\newtheorem{remark}[thm]{Remark}

\numberwithin{equation}{section}
\newcommand{\D}{\mathbb{D}}

\newcommand{\T}{\mathbb{T}}

\newcommand{\C}{\mathbb{C}}

\newcommand{\R}{\mathbb{R}}

\newcommand{\Z}{\mathbb{Z}}

\newcommand{\clb}{\mathcal{B}}

\newcommand{\cld}{\mathcal{D}}
\newcommand{\cle}{\mathcal{E}}

\newcommand{\clg}{\mathcal{G}}
\newcommand{\clh}{\mathcal{H}}
\newcommand{\clk}{\mathcal{K}}

\newcommand{\clr}{\mathcal{R}}
\newcommand{\cls}{\mathcal{S}}
\newcommand{\clt}{\mathcal{T}}
\newcommand{\clu}{\mathcal{U}}
\newcommand{\clw}{\mathcal{W}}

\newcommand{\raro}{\rightarrow}

\newcommand{\tv}{\tilde{V}}

\newcommand{\norm}[1]{\left\Vert#1\right\Vert}

\begin{document}

\title[Twisted isometries]{Orthogonal decompositions and twisted isometries}


\author{Narayan Rakshit}
\address{Narayan Rakshit, Statistics and Mathematics Unit, Indian Statistical Institute, 8th Mile, Mysore Road, Bangalore, Karnataka - 560059, India}
\email{narayan753@gmail.com}

\author{Jaydeb Sarkar}
\address{J. Sarkar, Indian Statistical Institute, Statistics and Mathematics Unit, 8th Mile, Mysore Road, Bangalore, 560059,
India}
\email{jay@isibang.ac.in, jaydeb@gmail.com}

\author{Mansi Suryawanshi}
\address{Mansi Suryawanshi, Statistics and Mathematics Unit, Indian Statistical Institute, 8th Mile, Mysore Road, Bangalore, Karnataka - 560059, India}
\email{mansisuryawanshi1@gmail.com}



\begin{abstract}
Let $n > 1$. Let $\{U_{ij}\}_{1 \leq i < j \leq n}$ be $\binom{n}{2}$ commuting unitaries on some Hilbert space $\mathcal{H}$, and suppose $U_{ji} := U_{ij}^*$, $1 \leq i < j \leq n$. An $n$-tuple of isometries $V = (V_1, \ldots ,V_n)$ on $\mathcal{H}$ is called $\mathcal{U}_n$-twisted isometry with respect to $\{U_{ij}\}_{i<j}$ (or simply $\mathcal{U}_n$-twisted isometry if $\{U_{ij}\}_{i<j}$ is clear from the context) if $V_i$'s are in the commutator $\{U_{st}: s \neq t\}'$, and $V_i^*V_j=U_{ij}^*V_jV_i^*$, $i \neq j$

\noindent We prove that each $\mathcal{U}_n$-twisted isometry admits a von Neumann-Wold type orthogonal decomposition, and prove that the universal $C^*$-algebra generated by $\mathcal{U}_n$-twisted isometries is nuclear. We exhibit concrete analytic models of $\mathcal{U}_n$-twisted isometries, and establish connections between unitary equivalence classes of the irreducible representations of the $C^*$-algebras generated by $\mathcal{U}_n$-twisted isometries and the unitary equivalence classes of the non-zero irreducible representations of twisted noncommutative tori. Our motivation of $\mathcal{U}_n$-twisted isometries stems from the classical rotation $C^*$-algebras and Heisenberg group $C^*$-algebras.
\end{abstract}

\maketitle

\tableofcontents

\section{Introduction}\label{sec:intro}

One of the most simple and fundamental of all the concepts studied in various branches of linear analysis, mathematical physics, and its related fields is the notion of isometries. Let $\clh$ be a Hilbert space (all Hilbert spaces in this paper are separable and over $\mathbb{C}$), and let $\clb(\clh)$ denote the $C^*$-algebra of all bounded linear operators on $\clh$. An operator $V \in \clb(\clh)$ is called \textit{isometry} if $V^*V = I_{\clh}$, or, equivalently, $\|Vh\| = \|h\|$ for all $h \in \clh$.

The typical examples are unitary operators, and shift operators. Recall that an isometry $V \in \clb(\clh)$ is called \textit{shift} if $V^{*m} \raro 0$ in the strong operator topology (that is, $\|V^{*m}h\| \raro 0$ as $m \raro \infty$ for all $h \in \clh$). The classical von Neumann–Wold decomposition theorem says that these are all examples of isometries:

\begin{thm}[J. von Neumann and H. Wold]\label{thm: classic}
Let $V \in \clb(\clh)$ be an isometry. Then $\clh = \clh_{\{1\}} \oplus \clh_{\emptyset}$ for some $V$-reducing closed subspaces $\clh_{\{1\}}$ and $\clh_{\emptyset}$ such that $V|_{\clh_{\{1\}}}$ is a shift and $V|_{\clh_{\emptyset}}$ is a unitary operator.
\end{thm}

In particular, $V = \begin{bmatrix}\text{shift} & 0 \\ 0 & \text{unitary} \end{bmatrix}$. This decomposition is canonical as well as unique in an appropriate sense. The von Neumann–Wold decomposition plays a central role in the foundation of linear operators; however, many of its variants are also studied in connection with $C^*$-algebras, ergodic theory, stochastic process, time series analysis and prediction theory, mathematical physics, etc. For instance, Theorem \ref{thm: classic} plays a key role in classifying $C^*$-algebras generated by isometries \cite{Coburn}. Another motivation for the study of isometries on Hilbert spaces, which is also relevant to our notion of twisted isometries, stems from the classical rotation algebras and Heisenberg group $C^*$-algebras \cite{Packer, Tron}. Also see \cite[Section 4]{W} in the context of universal $C^*$-algebras generated by pairs of isometries $V_1$ and $V_2$ such that
\[
V_1^* V_2 = e^{2 \pi i \vartheta} V_2 V_1^* \qquad (\vartheta \in \mathbb{R}).
\]
In this paper also, along with a von Neumann–Wold type decomposition, we present a few glimpses of applications of the above to $C^*$-algebras for a class of tuples of isometries (essentially, we will replace $e^{2 \pi i \vartheta}$ by a unitary $U$ in the commutator $\{V_1, V_2\}'$).

In view of Theorem \ref{thm: classic}, it is a natural question to ask whether an $n$-tuple, $n>1$, of isometries can be represented by tractable model operators as above. This is, on one hand, of course, almost hopeless in general, where, on the other extreme, pairs of commuting isometries represent (in an appropriate sense) the set of all bounded linear operators on Hilbert spaces. Nevertheless, Theorem \ref{thm: classic} motivates one to formulate the following definition:

\begin{definition}[Orthogonal decompositions]
Let $(V_1, \ldots, V_n)$ be an $n$-tuple of isometries acting on $\clh$. We  say that $V$ admits a von Neumann–Wold decomposition (orthogonal decomposition in short) if there exist $2^n$ closed subspaces $\{\clh_A\}_{A \subseteq I_n}$ of $\clh$ (some of them may be trivial) such that

(i) $\clh_A$ reduces $V_i$ for all $i=1, \ldots, n$, and $A \subseteq \{1,\ldots, n\}$,

(ii) $\clh = \bigoplus_{A \subseteq \{1,\ldots, n\}} \clh_A$, and

(iii) for each $A \subseteq \{1,\ldots, n\}$, ${V_i}|_{\clh_A}$, $i \in A$, is a shift, and ${V_j}|_{\clh_A}$, $j \in A^c$, is a unitary.
\end{definition}

We illustrate this with concrete examples: Let $z_{ij} \in \T$, $1 \leq i < j \leq n$, and suppose $z_{ji} = \bar{z}_{ij}$ for all $1 \leq i<j \leq n$. An $n$-tuple of isometries $(V_1, \ldots, V_n)$ on some Hilbert space $\clh$ is said to be \textit{doubly non-commuting isometries} if $V_i^* V_j = \bar{z}_{ij} V_j V_i^*$ for all $i \neq j$. The following comes from \cite[Theorem 3.6]{JP}:

\begin{thm}\label{thm: Pinto}
Each $n$-tuple of doubly non-commuting isometries admits an orthogonal decomposition.
\end{thm}

Note that if $z_{ij} = 1$, $i \neq j$, then doubly non-commuting isometries are simply doubly commuting isometries. Therefore, the above theorem recovers orthogonal decompositions of doubly commuting isometries \cite{S, Slo}. A question of obvious interest consists in enlarging the above class of tuples of isometries that admit the orthogonal decomposition. To address this question, we now introduce our primary object of study, twisted isometries on Hilbert spaces.

\begin{definition}[$\clu_n$-twisted isometries]\label{def: twisted isom}
Let $n > 1$. Let $\{U_{ij}\}_{1 \leq i < j \leq n}$ be $\binom{n}{2}$ commuting unitaries on a Hilbert space $\mathcal{H}$, and suppose $U_{ji} := U_{ij}^*$, $1 \leq i < j \leq n$. An $n$-tuple of isometries $(V_1, \ldots ,V_n)$ on $\clh$ is called $\clu_n$-twisted isometry with respect to $\{U_{ij}\}_{i<j}$ if
\begin{equation}\label{equation: u_ntwisted}
V_i^*V_j=U_{ij}^*V_jV_i^* \text{ and }V_kU_{ij}=U_{ij}V_k \qquad (i,j,k=1, \ldots, n, \text{ and }i \neq j).
\end{equation}
\end{definition}

Sometimes we will suppress the reference of the unitaries $\{U_{ij}\}_{1 \leq i < j \leq n}$ and simply say that $(V_1, \ldots ,V_n)$ is a $\clu_n$-twisted isometry. Also we must point out that the commutativity assumption on $\{U_{ij}\}_{1 \leq i < j \leq n}$ is automatic for our purpose (see Remark \ref{remark: Uij = V*V*VV}).

Clearly, doubly non-commuting isometries are also $\clu_n$-twisted isometries with respect to $\{z_{ij} I_{\clh}\}_{i < j}$. On the other hand, as we shall see in Section \ref{sec: example}, $\clu_n$-twisted isometries form a large class of $n$-tuples of isometries which also includes a number of interesting examples. In fact, Section \ref{sec: example} is the central part of this paper, while one of the central results of this paper is the following generalization of Theorem \ref{thm: Pinto} to the case of $\clu_n$-twisted isometries (see Theorem \ref{thm: main decomposition}).

\begin{theorem*}
Each $\clu_n$-twisted isometry admits an orthogonal decomposition.
\end{theorem*}

We wish to point out that our proof, even in this generality, is simpler than that of \cite{JP}. However, our proof also requires as background the classical von Neumann–Wold decomposition theorem.

\noindent Now we comment on the direct summands in the orthogonal decomposition of an isometry $V \in \clb(\clh)$ as in Theorem \ref{thm: classic}. One can easily prove \cite{S} that $\clh_{\{1\}}$ and $\clh_{\emptyset}$ in Theorem \ref{thm: classic} admits the following geometric representations
\begin{equation}\label{eqn: H_s and H_u}
\clh_{\{1\}} = \oplus_{j=0}^\infty V^j \clw \mbox{~and~} \clh_{\emptyset} = \cap_{j=0}^\infty V^j \clh,
\end{equation}
where $\clw = \ker V^*$. Moreover, the orthogonal decomposition in Theorem \ref{thm: classic} is unique in the following sense: Suppose $\cls_1$ and $\cls_2$ are reducing subspaces for $V$. If $V|_{\cls_1}$ is a shift, then $\cls_1 \subseteq \clh_{\{1\}}$. And, if $V|_{\cls_2}$ is a unitary, then $\cls_2 \subseteq \clh_{\emptyset}$. In particular, if $\cls_1 \oplus \cls_2 = \clh$, then $\cls_1 = \clh_{\{1\}}$ and $\cls_2 = \clh_{\emptyset}$.

\noindent In the setting of $\clu_n$-twisted isometries, we prove a similar geometric representation of each of the $2^n$ direct summands of the corresponding orthogonal decomposition. This is linked with the existence of the orthogonal decompositions (see Theorem \ref{thm: main decomposition}). Also we prove that the orthogonal decomposition is unique (see Corollary \ref{cor: unique}). These results form the subject of Section \ref{sec: main theorem}.

In Section \ref{sec: model and wand data}, we present analytic models of $\clu_n$-twisted isometries. Our model relies on two core concepts, namely, wandering subspaces and wandering data. We prove that the list of examples in Section \ref{sec: example} plays a pivotal role in the structure theory of $\clu_n$-twisted isometries.

In Theorem \ref{nuclearity}, we prove that the universal $C^*$-algebra generated by a $\clu_n$-twisted isometry, $n \geq 2$, is nuclear. This is the main content of Section \ref{sec: nuclear}.

Also, we intend with this paper to give a motivation for the study of (generalized) noncommutative tori associated with tuples of isometries, which is an analog of the classical anticommutation relations with unitary twists. However, here we will restrict ourselves to $\clu_n$-twisted isometries. For instance, in Section \ref{sec: Classification}, we introduce the twisted noncommutative tori for $\clu_n$-twisted isometries. Theorem \ref{thm: classification} states that the unitary equivalence classes of $\clu_n$-twisted isometries are in bijection with enumerations of $2^n$ unitary equivalence classes of unital representations of twisted noncommutative tori. In Corollary \ref{cor: classification of noncommutative tori}, we prove that the unitary equivalence classes of the non-zero irreducible representations of the $C^*$-algebras generated by $\clu_n$-twisted isometries are parameterized by the unitary equivalence classes of the non-zero irreducible representations of twisted $2^n$-tori.

Needless to say, the notion of $\clu_n$-twisted isometries is inspired by the earlier work on the classical rotation $C^*$-algebras and Heisenberg $C^*$-algebras at the level of unitaries \cite{AP, Packer, PR}. Some of our results are also motivated by \cite{JP}. However, on one hand, our results are more general, and on the other, our approach, even in the particular case of tuples of doubly non-commuting isometries, is significantly different and appears to be somewhat more natural.

Throughout the paper we follow the standard definition of unitarily equivalence: Two $n$-tuples $V = (V_1, \ldots, V_n)$ and $\tv = (\tv_1,\ldots, \tv_n)$ on Hilbert spaces $\clh$ and $\tilde{\clh}$, respectively, are said to be \textit{unitarily equivalent} if there exists a unitary $U: \clh \raro \tilde \clh$ such that $U V_i = \tv_i U$ for all $i=1, \ldots, n$. Also we use standard notation such as $\Z_+^n = \{k = (k_1, \ldots, k_n): k_i \in \Z_+\}$, $\mathbb{C}^n = \{z = (z_1, \ldots, z_n): z_i \in \mathbb{C}\}$, $z^k = z_1^{k_1} \cdots z_n^{k_n}$, and $V^k = V_1^{k_1} \cdots V_n^{k_n}$, whenever $k \in \Z_+^n$ and $V = (V_1, \ldots, V_n)$ on some Hilbert space.

\section{Examples}\label{sec: example}

This section introduces some basic concepts, and presents some (model) examples of $\clu_n$-twisted isometries. This also sets the stage for a more thorough treatment of $\clu_n$-twisted isometries in what follows. The present section is the central part of this paper.

Let $H^2(\D)$ denote the Hardy space over the unit disc $\D = \{z \in \mathbb{C}: |z| < 1\}$. Denote by $M_z$ the multiplication operator by the coordinate function $z$ on $H^2(\D)$, that is, $M_z f = zf$ for all $f \in H^2(\D)$. It is well known that $M_z$ is a shift of multiplicity one (as $\ker M_z^* = \mathbb{C}$). Now, let $H^2(\D^2)$ denote the Hardy space over the bidisc $\D^2$. Recall that $H^2(\D^2)$ is the Hilbert space of all square summable analytic functions on $\D^2$. That is, an analytic function $f(z) = \sum_{k \in \Z_+^2} \alpha_{k} z^{k}$ on $\D^2$ is in $H^2(\D^2)$ if and only if
\[
\|f\| := \Big(\sum_{k \in \Z_+^2} |a_{k}|^2\Big)^{\frac{1}{2}} < \infty.
\]
One can easily identify $H^2(\D^2)$ with $H^2(\D) \otimes H^2(\D)$ in a natural way: define $\tau : H^2(\D) \otimes H^2(\D) \raro H^2(\D^2)$ by $\tau(z^{k_1} \otimes z^{k_2}) = z_1^{k_1} z_2^{k_2}$, $k \in \Z_+^2$. Then $\tau$ is a unitary operator and
\[
\tau(M_z \otimes I_{H^2(\D)}) = M_{z_1} \tau \mbox{~and~} \tau(I_{H^2(\D)} \otimes M_z) = M_{z_2} \tau,
\]
where $M_{z_1}$ and $M_{z_2}$ are the multiplication operators by $z_1$ and $z_2$, respectively, on $H^2(\D^2)$. This construction works equally well for $H^2(\D^m)$, the Hardy space over $\D^m$, $m> 1$.

We are now ready for the main content of this section and begin with some elementary (but motivational) examples of $\clu_2$-twisted isometries.

\begin{example}\label{example 1}
It will be convenient to introduce a special class of diagonal operators parameterized by the circle group $\T$. For each $\lambda \in \T$, define (cf. \cite[proof of Lemma 1.2]{W})
\[
D[\lambda] z^m = \lambda^m z^m \qquad (m \in \Z_+).
\]
Clearly, $D[\lambda]$ is a unitary diagonal operator on $H^2(\D)$ and $D[\lambda]^* = D[\bar{\lambda}] =  \mbox{diag} (1, \bar{\lambda}, \bar{\lambda}^2, \ldots)$. It is easy to see that
\[
(M_z^* D[\lambda]) (z^m) =
\begin{cases}
\lambda^m z^{m-1} & \quad\text{if } m > 0
\\
0 &\quad \text{if } m=0,
\end{cases}
\]
and
\[
(D[\lambda] M_z^*) (z^m) =
\begin{cases}
\lambda^{m-1} z^{m-1} & \quad\text{if } m > 0
\\
0 &\quad \text{if } m=0,
\end{cases}
\]
and hence, $M_z^* D[\lambda] = \lambda D[\lambda] M_z^*$. Now we fix $\lambda \in \T$, and define $S_1$ and $S_2$ on $H^2(\D^2)$ by setting
\[
S_1 = M_z \otimes I_{H^2(\D)} \mbox{~and~} S_2 = D[\lambda] \otimes M_z.
\]
Therefore, $(S_1, S_2)$ is a pair of isometries on $H^2(\D^2)$, and $S_1^* S_2 = M_z^* D[\lambda] \otimes M_z$, and $S_2 S_1^* = D[\lambda] M_z^* \otimes M_z$. Then, $M_z^* D[\lambda] = \lambda D[\lambda] M_z^*$ implies $S_1^* S_2 = \lambda S_2 S_1^*$. We now consider the Hilbert space $\clh = H^2(\D^2) \oplus H^2(\D^2)$, and isometries $V_1= \mbox{diag}(S_1, S_2)$ and $V_2= \mbox{diag}(S_2, S_1)$ on $\clh$. If we set $U = \mbox{diag}(\bar{\lambda} I_{H^2(\D^2)}, {\lambda}I_{H^2(\D^2)})$, then
\[
V_1^* V_2 = \begin{bmatrix}S_1^* S_2 & 0
\\ 0 & S_2^* S_1
\end{bmatrix} = \begin{bmatrix}\lambda S_2S_1^* & 0 \\ 0 & \bar{\lambda} S_1S_2^*
\end{bmatrix}=
\begin{bmatrix}\lambda I_{H^2(\D^2)} & 0\\ 0 & \bar{\lambda}I_{H^2(\D^2)} \end{bmatrix}
V_2V_1^*,
\]
which implies that $V_1^* V_2 = U^* V_2 V_1^*$. Since $V_1, V_2 \in \{U\}'$, it follows that the pair $(V_1, V_2)$ is a (reducible) $\clu_2$-twisted isometry on $\clh$ with $\clu_2 = \{ U \}$.
\end{example}

Note that for each $\lambda \in \T$, the pairs $(M_z, D[\lambda])$ and $(S_1, S_2)$, defined as above, are doubly non-commuting isometries. This was considered and analyzed in the context of models of doubly noncommuting isometries in \cite{JP} (although their presentation is somewhat different than ours).

We continue and extend the discussion of Hardy space over $\D^m$, $m > 1$. For a Hilbert space $\cle$, we denote by $H^2_{\cle}(\D^m)$ the $\cle$-valued Hardy space over $\D^m$. Note that $H^2_{\cle}(\D^m)$ is the Hilbert space of all square summable analytic functions on $\D^m$ with coefficients in $\cle$. We simply set $H^2(\D^m) = H^2_{\mathbb{C}}(\D^m)$. In view of the natural identification
\[
z^{k} \eta \leftrightarrow  z^{k_1} \otimes \cdots \otimes z^{k_m} \otimes \eta \leftrightarrow z^{k} \otimes \eta \qquad (k \in \Z_+^m, \eta \in \cle),
\]
up to unitary equivalence, we have
\[
H^2_{\cle}(\D^m) = \underbrace{ H^2(\D) \otimes \cdots \otimes H^2(\D)}_{m-\mbox{times}} \; \otimes \; \cle = H^2(\D^m) \otimes \cle.
\]
In this setting, for each fixed $i =1, \ldots, m$, we also have (again, up to unitary equivalence)
\[
M_{z_i} = (I_{H^2(\D)} \otimes \cdots I_{H^2(\D)} \otimes \underbrace{M_{z}}_{i-\mbox{th}} \otimes I_{H^2(\D)} \otimes \cdots \otimes I_{H^2(\D)}) \otimes I_{\cle} = M_{z_i} \otimes I_{\cle},
\]
where $M_{z_i} f = z_i f$ for any $f$ either in $H^2_{\cle}(\D^m)$ or in $H^2(\D^m)$ (whichever is the case should be clear from the context). For simplicity, and whenever appropriate, we shall use the above identification interchangeably. Moreover, the above tensor product representations of the multiplication operators readily imply that $(M_{z_1},\ldots, M_{z_m})$ on $H^2_{\cle}(\D^m)$ is \textit{doubly commuting}, that is, $M_{z_i} M_{z_j} = M_{z_j} M_{z_i}$ and $M_{z_i}^* M_{z_p} = M_{z_p} M_{z_i}^*$ for all $i,j,p=1, \ldots, n$ and $i \neq p$.

We need to define another important notion before we proceed.

\begin{definition}\label{defn: index operator}
Let $j \in \{1, \ldots, m\}$. Given a Hilbert space $\cle$ and a unitary $U \in \clb(\cle)$, the \textit{$j$-th diagonal operator with symbol $U$}
is the unitary operator $D_j[U]$ on $H^2_{\cle}(\D^m)$ defined by
\[
D_j[U] (z^k \eta) = z^k (U^{k_j} \eta) \qquad (k \in \Z_+^m, \eta \in \cle).
\]
\end{definition}

We remind the reader that $k = (k_1, \ldots, k_m)$. In particular, if $m=1$ and $\cle = \mathbb{C}$, then $U$ is given by $U = \lambda$ for some $\lambda \in \T$, and then, as introduced earlier, $D_1[\lambda]$ is the diagonal operator $diag(1, \lambda, \lambda^2, \ldots)$ on $H^2(\D)$.

\begin{lem}\label{lemma:key index}
Let $\cle$ be a Hilbert space, and let $U$ and $\tilde{U}$ be commuting unitaries in $\clb(\cle)$. Suppose $i, j \in \{1, \ldots, n\}$. Then
\begin{enumerate}
\item $D_j[U]^* = D_j[U^*]$ and $D_i[U] D_j[\tilde{U}] = D_j[\tilde{U}] D_i[U]$.
\item $M_{z_i} D_j[U] = D_j[U] M_{z_i}$ whenever $i\neq j$.
\item $M_{z_i}^* D_i[U] = (I_{H^2(\D^n)} \otimes U)D_i[U] M_{z_i}^*$.
\end{enumerate}
\end{lem}
\begin{proof}
The first assertion follows from the definition of diagonal operators, and the commutativity of $U$ and $\tilde{U}$. To prove $(2)$, we assume that $k \in \Z_+^n$ and $\eta \in \cle$. Suppose $i \neq j$. We have on one hand
\[
(D_j[U] M_{z_i}) (z^k \eta) = D_j[U] (z^{k+e_i} \eta)  = z^{k+e_i} (U^{k_j}\eta),
\]
and on the other hand
\[
(M_{z_i} D_j[U]) (z^k \eta) = M_{z_i}(z^k (U^{k_j} \eta)) = z^{k+e_i} (U^{k_j} \eta),
\]
where $e_i$ denotes the element in $\Z_+^n$ with $1$ in the $i$-th slot and zero elsewhere. Here we used $i \neq j$ which implies that $k_j$ remains unchanged. For part $(3)$, we compute
\[
(M_{z_i}^* D_i[U]) (z^k \eta) = M_{z_i}^* (z^k U^{k_i} \eta) = \begin{cases}
z^{k-e_i} (U^{k_i} \eta) & \quad\text{if } k_i \neq 0
\\
0 &\quad \text{if } k_i =0.
\end{cases}
\]
On the other hand, since $D_i[U] (z^{k-e_i} \eta) = z^{k-e_i} (U^{k_i - 1} \eta)$ for $k_i \neq 0$, we have
\[
(D_i[U] M_{z_i}^*) (z^k \eta)
= \begin{cases}
z^{k-e_i} (U^{k_i - 1} \eta) & \quad\text{if } k_i \neq 0
\\
0 &\quad \text{if } k_i =0,
\end{cases}
\]
which completes the proof of part $(3)$.
\end{proof}

We now turn to more general examples of $\clu_n$-twisted isometries.  Let $\cle$ be a Hilbert space, and let $\{U_{ij}: i,j=1, \ldots, n, i\neq j\} \subseteq \clb(\cle)$ be a family of commuting unitaries. Suppose $U_{ji}:=U_{ij}^*$ for all $i \neq j$. Fix $m \in \{1, \ldots, n\}$. Consider $(n-m)$ unitary operators $\{ U_{m+1}, \ldots, U_{n}\}$ in $\clb(\cle)$. Suppose
\[
U_i U_j=U_{ij}U_jU_i \text{ and }U_iU_{pq}=U_{pq}U_i,
\]
for all $i,j = m+1 \ldots, n$, $i\neq j$, and $p,q=1,\ldots, n$, $p\neq q$. Set $M_1 = M_{z_1}$, and for each $2 \leq i \leq m$, define
\[
M_i =  M_{z_i} (D_1 [U_{i 1}] D_2 [U_{i 2}] \cdots D_{i-1}[U_{i i-1}]),
\]
and, for each $m+1 \leq j \leq n$, define
\[
M_j = (D_1[U_{j 1}] \cdots D_m[U_{j m}]) (I_{H^2(\D^m)} \otimes U_j).
\]
Then, by construction, $M = (M_1, \ldots, M_n)$ is an $n$-tuple of isometries on $H^2_{\cle}(\D^m)$. Moreover, $M$ is a $\clu_n$-twisted isometry with respect to $\{I_{H^2(\D^m)} \otimes U_{ij}\}_{i < j}$. This can be proved by repeated applications of Lemma \ref{lemma:key index}. For instance, if $1 < i < j$, then
\[
M_i^* M_j = (I_{H^2(\D^m)} \otimes U_{ij})^* M_j M_i^*,
\]
follows from the fact that $M_{z_i}^* M_{z_j} = M_{z_j} M_{z_i}^*$, and, notably, from part $(3)$ of Lemma \ref{lemma:key index} that $M_{z_i}^* D_i[U_{ji}] = (I_{H^2(\D^m)} \otimes U_{ji})D_i[U_{ji}] M_{z_i}^*$. We summarize this with the following proposition:

\begin{prop}
Let $\cle$ be a Hilbert space, and let $\{U_{ij}: i,j = 1, \ldots, n, i \neq j\}$ be a commuting family of unitaries on $\cle$ such that $U_{ji}:=U_{ij}^*$ for all $i \neq j$. Fix $m \in \{1, \ldots, n\}$ and consider $(n-m)$ unitary operators $\{ U_{m+1}, \ldots, U_{n}\}$ in $\clb(\cle)$ such that
\[
U_i U_j=U_{ij}U_jU_i \text{ and }U_iU_{pq}=U_{pq}U_i,
\]
for all $m+1\leq i\neq j \leq n$, and $1\leq p\neq q \leq n$. Let $M_1 = M_{z_1}$ and
\[
M_i =
\begin{cases}
M_{z_i} \Big(D_1 [U_{i 1}] D_2 [U_{i 2}] \cdots D_{i-1}[U_{i i-1}]\Big) & \mbox{if } 2 \leq i \leq m
\\
\Big(D_1[U_{i 1}] \cdots D_m[U_{i m}]\Big) \Big(I_{H^2(\D^m)} \otimes U_i\Big) & \mbox{if } m+1 \leq i \leq n.
\end{cases}
\]
Then $M_1,\ldots,M_m$ are shifts, $M_{m+1},\ldots,M_n$ are unitaries, and $(M_1, \ldots, M_n)$ is a $\clu_n$-twisted isometry on $H^2_{\cle}(\D^m)$ with respect to $\{I_{H^2(\D^m)} \otimes U_{ij}\}_{i < j}$.
\end{prop}

We will return to this in the context of analytic models and complete unitary invariants in Sections \ref{sec: model and wand data} and \ref{sec: invariants}, respectively.

\section{Orthogonal decompositions}\label{sec: main theorem}

The principal goal of this section is to prove that $\clu_n$-twisted isometries admit orthogonal decomposition. We begin by fixing some notations (once again, we stress that $n > 1$).

\begin{enumerate}
\item $I_n = \{1, \ldots, n\}$. $A = \{i_1, \ldots, i_m\} \subseteq I_n$, $i_1 < \cdots < i_m$, whenever $A \neq \emptyset$.
\item If $V = (V_1, \ldots, V_n)$, then $V_A = (V_{i_1}, \ldots, V_{i_m})$ whenever $A = \{i_1, \ldots, i_m\} \subseteq I_n$.
\item $V_A^{k} = V_{i_1}^{k_1} \cdots V_{i_m}^{k_m}$ whenever $k = (k_1, \ldots, k_m) \in \Z_+^m$ and $A =\{i_1, \ldots, i_m\} \subseteq I_n$.
\item $\clw_A = \bigcap_{i \in A} \ker V_i^*$ for all non-empty $A \subseteq I_n$, $\clw_{\emptyset} := \clh$, and $|\emptyset| := 0$.
\end{enumerate}

The following result essentially says that $\clu_n$-twisted isometries are ``twisted doubly commuting'' (see \cite[page 2671]{Pelle} for the scalar case).

\begin{lem}\label{lemma: commute}
Let $U$ be a unitary and $(V_1, V_2)$ be a pair of isometries on $\clh$. Suppose $V_1, V_2 \in \{U\}'$ and $V_1^* V_2 = U^* V_2 V_1^*$. Then $V_1 V_2 = U V_2 V_1$.
\end{lem}
\begin{proof}
If we denote $X = V_1 V_2 - U V_2 V_1$, then
\[
X^* X = (V_2^* V_1^* - U^* V_1^* V_2^*)(V_1 V_2 - U V_2 V_1) = 2I - U V_2^* V_1^* V_2 V_1 - U^* V_1^* V_2^* V_1 V_2.
\]
Using $V_1^*V_2 = U^* V_2 V_1^*$, one easily verifies that $U V_2^* V_1^* V_2 V_1 = U^* V_1^* V_2^* V_1 V_2 = I$. This completes the proof that $X^*X = 0$ and hence $V_1 V_2 = U V_2 V_1$.
\end{proof}

In particular, if $(V_1, \ldots, V_n)$ is a $\clu_n$-twisted isometry, then $V_iV_j = U_{ij} V_j V_i$ for all $i \neq j$. We note that the converse of the above lemma is not true \cite{Pelle}.

\begin{remark}\label{remark: Uij = V*V*VV}
The commutativity assumption of $\{U_{ij}\}_{1 \leq i < j \leq n}$ in the definition of $\clu_n$-twisted isometries (see Definition \ref{def: twisted isom}) is automatic in the following sense: Let $\{U_{ij}\}_{1 \leq i < j \leq n}$ be an $\binom{n}{2}$-tuple of unitaries on $\clh$, and let $(V_1, \ldots ,V_n)$ be an $n$-tuple of isometries on $\clh$. Let $U_{ji} := U_{ij}^*$ for all $1 \leq i < j \leq n$, and suppose $V_p \in \{U_{st}: s \neq t\}'$ for all $p=1, \ldots, n$. Then $U_{ij} U_{st} = U_{st} U_{ij}$ for all $i \neq j$ and $s \neq t$. Indeed, we first observe that $V_i^* V_j = U^*_{ij} V_j V_i^*$ and $V_i, V_j \in \{U_{st}: s \neq t\}'$ implies
\begin{equation}\label{eqn: U_{ij} = V_i^* V_j^* V_i V_j}
U_{ij} = V_i^* V_j^* V_i V_j \qquad (i \neq j).
\end{equation}
Hence we obtain $U_{ij} U_{st} = (V_i^* V_j^* V_i V_j)U_{st} = U_{st} U_{ij}$.
\end{remark}

The following elementary lemmas will play an important role. Throughout these lemmas, $V = (V_1, \ldots, V_n)$ will be a $\clu_n$-twisted isometry, and $A \subseteq I_n$. We begin with reducibility of wandering subspaces.

\begin{lem}\label{lemma: reducing}
$\clw_A$ reduces $V_j$ for all $j \in A^c$.
\end{lem}
\begin{proof}
Suppose $\eta \in \clw_A$, that is, $V_i^{\ast}\eta = 0$ for all $i \in A$. Suppose $j \notin A$. Since $V_i^{\ast}(V_j\eta)=U^*_{ij} V_jV_i^{\ast}\eta = 0$, we have $V_j \clw_A \subseteq \ker V_i^{\ast}$ for all $i\in A$. Thus $V_j \clw_A \subseteq \clw_A$. Also observe that by Lemma \ref{lemma: commute}, we have $V_i^* V_j^* = U_{ij} V_j^* V_i^*$, and hence, as before, $V_j^* \clw_A \subseteq \clw_A$.
\end{proof}

In particular, $V_j|_{\clw_A}$ is an isometry on $\clw_A$. It is now natural to examine $\ker (V_j|_{\clw_A})^*$. Evidently, $\ker (V_j|_{\clw_A})^* = \clw_A \ominus V_j \clw_A$.

\begin{lem}\label{lemma: wandering}
$\clw_A \ominus V_j \clw_A = \clw_{A \cup \{j\}}$ for all $j \in A^c$.
\end{lem}
\begin{proof}
The goal is to show that $\clw_A \ominus V_j \clw_A = \clw_{A} \cap \clw_j$. Indeed, this follows from Lemma \ref{lemma: reducing}: $\clw_A$ reduces $V_j$, and hence $V_j = \mbox{diag}(V_j|_{\clw_A}, V_j|_{\clw_A^{\perp}})$ on $\clh = \clw_A \oplus \clw_A^{\perp}$.
\end{proof}

We now turn to the reducibility property of wandering subspaces of corresponding unitary operators.

\begin{lem}\label{lemma: W reducing U}
$\clw_A$ reduces $U_{ij}$, and $U_{ij} \clw_A = \clw_A$ for all $i \neq j$.
\end{lem}
\begin{proof}
For the first part, we note that $V_k U_{ij} = U_{ij} V_k$ and hence $V_k U_{ij}^* = U_{ij}^* V_k$ for all $i\neq j$ and $k$. Then for each $\eta \in \clw_A$ and $k \in A$, we have
\[
V_k^* U_{ij} \eta = U_{ij} V_k^* \eta = 0,
\]
and similarly, $V_k^* U_{ij}^* \eta =0$. The latter assertion is trivial, as $U_{ij}|_{\clw_A}$ is a unitary.
\end{proof}

Now we are ready to prove the orthogonal decomposition theorem. We will use the following convention consistently: For each $A \subseteq I_n$, we set $\Z_+^{|A|} = \emptyset$ if $A = \emptyset$, and we denote by $\Z_+^{|A|}$ the set of $|A|$-tuples of elements of $\Z_+$ whenever $A \neq \emptyset$.

\begin{thm}\label{thm: main decomposition}
Let $V = (V_1,\ldots,V_n)$ be a $\clu_n$-twisted isometry on $\clh$. Then $V$ admits an orthogonal decomposition $\clh = \bigoplus\limits_{A \subseteq I_n} \clh_A$, where
\[
\clh_A = \bigoplus\limits_{k \in \Z_+^{|A|}} V_A^k \Big( \bigcap\limits_{l \in \Z_+^{n - |A|}} V_{I_n \setminus A}^l \clw_A\Big) \qquad (A \subseteq I_n).
\]
\end{thm}
\begin{proof} We will prove this by induction. Suppose $(V_1,\ldots,V_n)$ is a $\clu_n$-twisted isometry on $\clh$. Set $V(m) = (V_1, \ldots, V_m)$, $2 \leq m \leq n$. We shall first prove our assertion when $m=2$. Let us denote $\clw_i = \clw_{\{i\}}$. Using Theorem \ref{thm: classic} (also \eqref{eqn: H_s and H_u}) applied to $V_1$ on $\clh$, we find
\[
\clh= (\oplus_{k_1\in \Z_+} V_1^{k_1} \clw_1) \oplus (\cap_{k_1\in \Z_+} V_1^{k_1} \clh).
\]
Note that, by Lemma \ref{lemma: reducing}, $\clw_1$ reduces $V_2$. Then, by applying Theorem \ref{thm: classic} to the isometry $V_2|_{\clw_1}$, we obtain the orthogonal decomposition
\[
\clw_1 = (\oplus_{k_2\in \Z_+} V_2^{k_2} (\clw_1 \ominus V_2 \clw_1)) \oplus (\cap_{k_2\in \Z_+}(V_2^{k_2} \clw_1)).
\]
Now by Lemma \ref{lemma: wandering} we have $\clw_1 \ominus V_2 \clw_1 = \clw_{\{1,2\}}$, and hence
\begin{equation}\label{eqn: V(2)}
\clh = \Big[\bigoplus\limits_{k_1,k_2\in \Z_+} V_1^{k_1}V_2^{k_2} \clw_{\{1,2\}}\Big] \bigoplus \Big[ \bigoplus\limits_{k_1\in \Z_+} V_1^{k_1} (\bigcap\limits_{k_2\in \Z_+} V_2^{k_2} \clw_1)\Big] \bigoplus \Big[ \bigcap\limits_{k_1\in \Z_+} V_1^{k_1} \clh\Big].
\end{equation}
Note that the restrictions of $V_1$ and $V_2$ to the first and the second summands are shifts, and shift and unitary, respectively, and the restriction of $V_1$ to the third summand is a unitary. Now, applying Theorem \ref{thm: classic} (and the representations in \eqref{eqn: H_s and H_u}) to $V_2$ on $\clh$, we obtain
\[
\clh = (\oplus_{k_2\in \Z_+} V_2^{k_2} \clw_2) \oplus (\cap_{k_2\in \Z_+} V_2^{k_2}\clh).
\]
By Lemma \ref{lemma: commute}, we have $V_1^{k_1} V_2^{k_2} = U_{12}^{k_1 + k_2} V_2^{k_2} V_1^{k_1}$ for all $k_1, k_2 > 0$. Lemma \ref{lemma: W reducing U} then implies that $V_1^{k_1} V_2^{k_2} \clw_2 = V_2^{k_2} V_1^{k_1} \clw_2$ for all $k_1, k_2 > 0$. Therefore
\[
V_1^{k_1}\clh = (\bigoplus\limits_{k_2\in \Z_+} V_1^{k_1} V_2^{k_2} \clw_2) \bigoplus (\bigcap\limits_{k_2\in \Z_+} V_1^{k_1} V_2^{k_2}\clh) = (\bigoplus\limits_{k_2\in \Z_+} V_2^{k_2} V_1^{k_1} \clw_2) \bigoplus (\bigcap\limits_{k_2\in \Z_+} V_1^{k_1} V_2^{k_2}\clh),
\]
for all $k_1 \in \Z_+$, from which it follows that
\[
\bigcap\limits_{k_1\in \Z_+} V_1^{k_1} \clh = (\bigoplus\limits_{k_2\in \Z_+} V_2^{k_2} (\bigcap_{k_1 \in \Z_+}V_1^{k_1} \clw_2)) \bigoplus (\bigcap\limits_{k_1,k_2\in \Z_+} V_1^{k_1} V_2^{k_2}\clh).
\]
We can then rewrite \eqref{eqn: V(2)} as $\clh = \oplus_{A \subseteq I_2} \clh_A$. This yields an orthogonal decomposition of the pair $V(2)$. Now suppose that $V(m)$, $m < n$, admits the orthogonal decomposition $\clh =\oplus_{A \subseteq I_m}\clh_A$, where
\[
\clh_A = \oplus_{k_a \in \Z_+^{|A|}} V_A^{k_a} (\cap_{k_c \in \Z_+^{m - |A|}}V_{I_m\setminus A}^{k_c} \clw_A).
\]
Recall that by convention, $\clw_{\emptyset} = \clh$ and $|\emptyset| = 0$. Since, by Lemma \ref{lemma: reducing}, $V_{m+1}$ reduces $\clw_A$, by applying Theorem \ref{thm: classic} to the isometry $V_{m+1}|_{\clw_A}$, and noting, by virtue of Lemma \ref{lemma: wandering}, that $\clw_A \cap \clw_{m+1} = \clw_{A \cup \{m+1\}}$, we obtain
\[
\clw_A= (\oplus_{j_{m+1}\in \Z_+} V_{m+1}^{j_{m+1}} \clw_{A \cup \{m+1\}}) \oplus (\cap_{j_{m+1}\in \Z_+} V_{m+1}^{j_{m+1}}\clw_A).
\]
This implies that
\[
\begin{split}
\clh_A & = \bigoplus\limits_{k_a \in \Z_+^{|A|}}  V_A^{k_a} \Big[\bigcap\limits_{k_c \in \Z_+^{m - |A|}} V_{I_m\setminus A}^{k_c} \Big(\bigoplus\limits_{j_{m+1}\in \Z_+} V_{m+1}^{j_{m+1}} \clw_{A \cup \{m+1\}} \bigoplus \bigcap\limits_{j_{m+1}\in \Z_+} V_{m+1}^{j_{m+1}} \clw_A \Big)\Big]
\\
& = \bigoplus\limits_{k_a \in \Z_+^{|A|}}  V_A^{k_a} \Big[\bigcap\limits_{k_c \in \Z_+^{m - |A|}} V_{I_m\setminus A}^{k_c} \Big(\bigoplus\limits_{j_{m+1}\in \Z_+} V_{m+1}^{j_{m+1}} \clw_{A \cup \{m+1\}}\Big) \bigoplus \Big(\bigcap\limits_{\substack{k_c \in \Z_+^{m - |A|} \\ j_{m+1}\in \Z_+}} V_{I_m\setminus A}^{k_c} V_{m+1}^{j_{m+1}} \clw_A \Big)\Big].
\end{split}
\]
By Lemma \ref{lemma: commute}, for each non-zero  $j_{m+1} \in \Z_+$ and $k_c \in \Z_+^{m-|A|}$, there exists a monomial $P_{j_{m+1}, k_c} \in \mathbb{C}[z_1, \ldots, z_{\binom{n}{2}}]$ such that
\[
V_{m+1}^{j_{m+1}}V_{I_m \setminus A}^{k_c} = P_{j_{m+1}, k_c}(U)V_{I_m \setminus A}^{k_c}V_{m+1}^{j_{m+1}}.
\]
Evidently, $P_{j_{m+1}, k_c}(U)$ is a monomial in $\{U_{ij}\}_{i<j}$. By Lemma \ref{lemma: W reducing U},
\[
V_{m+1}^{j_{m+1}}V_{I_m \setminus A}^{k_c}\clw_{A \cup \{m+1\}}= V_{I_m \setminus A}^{k_c}V_{m+1}^{j_{m+1}}\clw_{A \cup \{m+1\}},
\]
for all $j_{m+1} \in \Z_+$ and $k_c \in \Z_+^{m-|A|}$, and hence
\[
V_{I_m\setminus A}^{k_c} \Big(\bigoplus\limits_{j_{m+1}\in \Z_+} V_{m+1}^{j_{m+1}} \clw_{A \cup \{m+1\}}\Big) = \bigoplus_{j_{m+1}\in \Z_+} V_{m+1}^{j_{m+1}} \Big(V_{I_m\setminus A}^{k_c} \clw_{A \cup \{m+1\}} \Big),
\]
for all $k_c \in \Z_+^{m - |A|}$. Therefore
\[
\clh_A = \Big[\bigoplus_{\substack{k_a \in \Z_+^{|A|} \\ j_{m+1}\in \Z_+}} V_A^{k_a} V_{m+1}^{j_{m+1}} \Big(\bigcap\limits_{k_c \in \Z_+^{m - |A|}} V_{I_m\setminus A}^{k_c} \clw_{\tilde{A}} \Big)\Big]
\bigoplus \Big[\bigoplus\limits_{k_a \in \Z_+^{|A|}} V_A^{k_a} \Big(\bigcap_{\substack{k_c \in \Z_+^{m- |A|} \\ j_{m+1}\in \Z_+}} V_{I_m\setminus A}^{k_c} V_{m+1}^{j_{m+1}} \clw_A\Big)\Big],
\]
where $\tilde{A} = {A \cup \{m+1\}}$. This implies $\clh = \oplus_{A \subseteq I_{m+1}}\clh_A$, and hence $V(m+1)$ admits the orthogonal decomposition. This completes the proof.
\end{proof}

We will outline an alternate viewpoint of the above proof at the end of Section \ref{sec: model and wand data}.

In the remainder of this section, we discuss the uniqueness of the above orthogonal decomposition. Let $V = (V_1, \ldots, V_n)$ be a $\clu_n$-twisted isometry on $\clh$, $A \subseteq I_n$, and let a closed subspace $\cls \subseteq \clh$ reduce $V$. Suppose $V_i|_{\cls}$, $i \in A$, is a shift, and $V_j|_{\cls}$, $j \in A^c$, is a unitary. Set $\tilde V_i = V_i|_{\cls}$, $i \in I_n$. Now \eqref{eqn: U_{ij} = V_i^* V_j^* V_i V_j} implies that $\cls$ reduces $U_{ij}$, $i \neq j$. Then $\tilde{U}_{ji}=\tilde{U}_{ij}^*$ for all $1\leq i<j\leq n$, where $\tilde{U}_{ij} = U_{ij}|_{\cls}$, $i \neq j$. Evidently, $\tilde V :=(\tilde V_1, \ldots, \tilde V_n)$ on $\cls$ is a ${\clu}_n$-twisted isometry with respect to $\{\tilde{U}_{ij} \}_{i\neq j}$. Applying Theorem \ref{thm: main decomposition} to $\tilde V$, we obtain the orthogonal decomposition of $\tilde V$ as $\cls = \oplus_{B \subseteq I_n} \clh_B$. We claim that $\clh_B = \{0\}$ for all $B \neq A$, $B \subseteq I_n$. To see this, we first write $\tilde{\clw}_B = \cap_{i \in B} \ker \tilde V_i^*$, $B \subseteq I_n$. Let $i \in B \setminus A$. Then  $\tilde{V}_i = V_i|_{\cls}$ is a unitary, and hence $\tilde{\clw}_B = \{0\}$, which implies $\clh_B = \{0\}$. Now assume that $i \in A \setminus B$. Then $V_i|_{\clh_B}$ is a unitary, where on the other hand, $i \in A$ implies that $\tilde{V}_i$ is a shift, and hence $V_i|_{\clh_B}$ is a shift. This contradiction again shows that $\clh_B = \{0\}$. Thus
\[
\cls = \bigoplus\limits_{k \in \Z_+^{|A|}} \tilde V_A^k \Big( \bigcap\limits_{l \in \Z_+^{n - |A|}} \tilde V_{I_n \setminus A}^l \tilde \clw_A\Big).
\]
Again, by convention, we define $\tilde \clw_{\emptyset} = \cls$, $\clw_{\emptyset} = \clh$, and $|\emptyset| = 0$. Now, on the other hand, we have $\tilde \clw_A \subseteq \clw_A$. This simply follows from the fact that $\cls$ reduces the tuple $V$, and $\ker (V_i|_{\cls})^* = \ker V_i^*|_{\cls} \subseteq \ker V_i^*$ for all $i\in A$. Lemma \ref{lemma: reducing} then implies that $\tilde \clw_A$ reduces $V_i$, $i \notin A$, and hence $\cap_{l \in \Z_+^{n - |A|}} \tilde V_{I_n \setminus A}^l \tilde \clw_A \subseteq \cap_{l \in \Z_+^{n - |A|}} V_{I_n \setminus A}^l \clw_A$. Then
\[
\cls = \bigoplus\limits_{k \in \Z_+^{|A|}} \tilde V_A^k \Big( \bigcap\limits_{l \in \Z_+^{n - |A|}} \tilde V_{I_n \setminus A}^l \tilde \clw_A\Big) \subseteq \bigoplus\limits_{k \in \Z_+^{|A|}} V_A^k \Big( \bigcap\limits_{l \in \Z_+^{n - |A|}} V_{I_n \setminus A}^l \clw_A\Big) = \clh_A.
\]
This proves the nontrivial implication of the following proposition.

\begin{prop}\label{prop: unique}
Let $(V_1,\ldots, V_n)$ be a $\clu_n$-twisted isometry on $\clh$, $\cls$ be a closed $V$-reducing subspace of $\clh$, and let $A \subseteq I_n$. Suppose
\[
\clh_A := \bigoplus_{k \in \Z_+^{|A|}} V_A^k (\bigcap_{l \in \Z_+^{n - |A|}} V_{I_n \setminus A}^l \clw_A).
\]
Then the following are equivalent.

(i) $V_i|_{\cls}$ is a shift and $V_j|_{\cls}$ is a unitary for each $i \in A$ and $j \in A^c$, respectively.

(ii) $\cls \subseteq \clh_A$.

(iii) $P_{\cls} P_{\clh_A} = P_{\cls}$.
\end{prop}
\begin{proof}
(ii) $\Leftrightarrow$ (iii) is a general fact. (i) $\Rightarrow$ (ii) follows from the preceding computation, while (ii) $\Rightarrow$ (i) is straightforward.
\end{proof}

One may compare the above statement with the second part of \cite[Theorem 3.4]{JP}. The above proposition also yields the uniqueness part of the orthogonal decomposition.

\begin{cor}\label{cor: unique}
Let $V = (V_1,\ldots,V_n)$ be a $\clu_n$-twisted isometry on $\clh$, and set \[
\clh_A := \bigoplus_{k \in \Z_+^{|A|}} V_A^k (\bigcap_{l \in \Z_+^{n - |A|}} V_{I_n \setminus A}^l \clw_A) \qquad (A \subseteq I_n).
\]
Let $\cls_A$, $A \subseteq I_n$, be a $V$-reducing closed subspace of $\clh$.  Let $\clh = \oplus_{A\subseteq I_n} \cls_A$, and suppose $V_i|_{\cls_A}$ is a shift and $V_j|_{\cls_A}$ is a unitary for each $i \in A$ and $j \in A^c$, respectively. Then $\cls_A = \clh_A$ for all $A \subseteq I_n$.
\end{cor}
\begin{proof}
This immediately follows from (i) $\Rightarrow$ (ii) of Proposition \ref{prop: unique}.
\end{proof}

\section{Analytic models and wandering data}\label{sec: model and wand data}

In this section, we describe models of $\clu_n$-twisted isometries. Actually, we prove that the examples in Section \ref{sec: example} are the basic ``building blocks'' of $\clu_n$-twisted isometries.

Recall that one of the most important components of the classical von Neumann-Wold decomposition theorem is the separation of the shift part (if any) from a given isometry. One of the main points, therefore, is to find a canonical method of separating shifts (if any) from tuples of isometries. An additional benefit also arises here since a shift operator can be represented as the multiplication operator by the coordinate function $z$ on some (canonical) vector-valued Hardy space over $\D$. This is also the basic theme in all other related orthogonal decompositions of (tuples of) operators. For instance, suppose $V \in \clb(\clh)$ is an isometry. By \eqref{eqn: H_s and H_u}, the orthogonal decomposition of the $1$-tuple $V = (V)$ is given by $\clh = \clh_{\{1\}} \oplus \clh_{\emptyset}$, where $\clh_{\{1\}} = \oplus_{j=0}^\infty V^j \clw$ and $\clh_{\emptyset} = \cap_{j=0}^\infty V^j \clh$, and $\clw = \ker V^*$. Define the \textit{canonical unitary} $\Pi_V: \clh_{\{1\}} \raro H^2_{\clw}(\D)$ by $\Pi_V (V^m \eta) = z^m \eta$, $m \in \Z_+$, $\eta \in \clw$. Then
\begin{equation}\label{eqn:Pi_V V = M_z Pi_V}
(\Pi_V \oplus I_{\clh_{\emptyset}})(V|_{\clh_{\{1\}}} \oplus V|_{\clh_{\emptyset}}) = (M_z \oplus V|_{\clh_{\emptyset}}) (\Pi_V \oplus I_{\clh_{\emptyset}}).
\end{equation}
It then follows that $V$ on $\clh$ is unitarily equivalent to $M_z \oplus V|_{\clh_{\emptyset}}$ on $H^2_{\clw}(\D) \oplus \clh_{\emptyset}$. In other words, the shift part of $V$ admits an analytic representation in terms of the multiplication operator $M_z$ on the $\clw$-valued Hardy space over $\D$. It is also worthwhile to recall that $\mbox{dim} \clw$ is the only unitary invariant of the shift $M_z$ on $H^2_{\clw}(\D)$.

With the above motivation in mind, we now return to $\clu_n$-twisted isometries. First of all, following \cite[Definition 3.7]{JP}, we introduce two core concepts:

\begin{definition}\label{def: A wandering sub}
For a $\clu_n$-twisted isometry $V = (V_1, \ldots, V_n)$ on a Hilbert space $\clh$, and for each $A \subseteq I_n$, the $A$-wandering subspace of $V$ is defined by
\[
\cld_A(V) = \bigcap\limits_{l \in \Z_+^{n - |A|}} V_{I_n \setminus A}^l \clw_A.
\]
Moreover, if $A^c =\{q_1, \ldots, q_{n-m}\}$, then the $(n-m+1)$-tuple
\[
wd_V(A) =\left(1_{\cld_A(V)},V_{q_1}|_{\cld_A(V)},\ldots,V_{q_{n-m}}|_{\cld_A(V)}\right),
\]
on $\cld_A(V)$ is called the $A$-wandering data of $V$.
\end{definition}

We often denote $\cld_A(V)$ as $\cld_A$ if $V$ is clear from the context. Note that the following lemma ensures that the $A$-wandering data $wd_V(A)$ is a well-defined $(n-m+1)$-tuple on $\cld_A$.

\begin{lem}\label{lemma: D_A reduces V_j}
$\cld_A$ reduces $V_j$ and $U_{st}$, and $U_{st} \cld_A = \cld_A$ for all $j \in A^c$ and $s \neq t$.
\end{lem}
\begin{proof}
Suppose $A = \emptyset$. Then $\clw_A = \clh$, by convention, and hence $\cld_A = \clh_A$, by Theorem \ref{thm: main decomposition}, which reduces $V_j$ for all $j \in I_n$. If $A = I_n$, then $\cld_A = \clw_{I_n}$, and the statement is nothing but Lemma \ref{lemma: reducing} and Lemma \ref{lemma: W reducing U}. Suppose $A = \{p_1, \ldots, p_m\}$ for some $1 \leq m < n$, and suppose $j \in A^c$. Observe that
\begin{equation}\label{eqn: V_j V_i^m}
V_p V_q^i = U^i_{pq} V_q^i V_p,
\end{equation}
for all $p \neq q$ and $i \in \Z_+$. This essentially follows from Lemma \ref{lemma: commute} and the fact that $V_p, V_q \in \{U_{pq}\}'$. If $A^c=\{j,q_1,\ldots,q_{n-m-1}\}$ then $V_j V^l_{I_n \setminus A} \clw_A = V^l_{I_n \setminus A} V_j (U^{l_1}_{jq_1} \cdots U^{l_{n-m-1}}_{j q_{n-m-1}}) \clw_A$ for all $l \in \Z_+^{n-m-1}$. By Lemma \ref{lemma: W reducing U} and then by Lemma \ref{lemma: reducing}, it follows that
\[
V_j V^l_{I_n \setminus A} \clw_A = V^l_{I_n \setminus A} V_j \clw_A \subseteq V^l_{I_n \setminus A} \clw_A,
\]
and hence $V_j \cld_A \subseteq \cld_A$. Similarly, we have $V_j^* \cld_A \subseteq \cld_A$, and hence $\cld_A$ reduces $V_j$. The remaining part simply follows from the first part and \eqref{eqn: U_{ij} = V_i^* V_j^* V_i V_j}.
\end{proof}

Let $V = (V_1,\ldots,V_n)$ be a $\clu_n$-twisted isometry on a Hilbert space $\clh$. Theorem \ref{thm: main decomposition} then implies that $\clh = \oplus_{A \subseteq I_n} \clh_A$, where $\clh_A=\oplus_{k\in \Z_+^{|A|}}V_A^k \cld_A$, and $V_i|_{\clh_A}$ is a shift and $V_j|_{\clh_A}$ is a unitary for each $i \in A$ and $j \in A^c$, respectively, and $A \subseteq I_n$. In view of the discussion preceding Definition \ref{def: A wandering sub}, it is natural to investigate the possibility of carrying over the analytic construction of the shift part of $V|_{\clh_A}$, $A \subseteq I_n$. Of course, the restriction of $V$ to $\clh_{\emptyset} = \cap_{k \in \Z_+^n} V^k \clh$ is a unitary tuple. We now examine the case where $A \neq \emptyset$.

Let $A = \{p_1, \ldots, p_m\} \subseteq I_n$ for some $m \geq 1$, and suppose $\clh_A \neq \{0\}$ (or, equivalently, $\cld_A \neq \{0\}$). In view of the orthogonal decomposition $\clh_A = \oplus_{k\in \Z_+^m}V_A^k \cld_A$ and \eqref{eqn:Pi_V V = M_z Pi_V}, we have the canonical unitary $\pi_A: \clh_A \raro H^2_{\cld_A}(\D^m)$, where (note that $m = |A| >0$)
\begin{equation}\label{eqn: Pi V_A}
\pi_A (V_A^k \eta) = z^k \eta \qquad (k \in \Z_+^m, \eta \in \cld_A).
\end{equation}
Suppose $k \in \Z_+^m$ and $\eta \in \cld_A$. We then get
\[
(\pi_A V_{p_1} \pi_A^*)(z^k \eta) = \pi_A(V_{p_1} V^k_A \eta) = \pi_A (V_{p_1}^{k_1+1} V_{p_2}^{k_2} \cdots V_{p_m}^{k_m} \eta) = z_1 (z^k \eta),
\]
that is, $\pi_A V_{p_1} = M_{z_1} \pi_A$. Next, assume that $1 < i \leq m$. By \eqref{eqn: V_j V_i^m}, we know that
\[
V_{p_i} V^k_A = V_{p_i} (V_{p_1}^{k_1} \cdots V_{p_m}^{k_m}) =  V_{p_1}^{k_1} \cdots V_{p_{i-1}}^{k_{i-1}} V_{p_i}^{k_i + 1} V_{p_{i+1}}^{k_{i+1}} \cdots V_{p_m}^{k_m} (U_{p_i p_1}^{k_1} \cdots U_{p_i p_{i-1}}^{k_{i-1}}),
\]
and $U_{p_i p_1}^{k_1} \cdots U_{p_i p_{i-1}}^{k_{i-1}} \eta \in \cld_A$, by Lemma \ref{lemma: D_A reduces V_j}. Hence
\[
(\pi_A V_{p_i} \pi_A^*)(z^k \eta) = \pi_A(V_{p_i}V^k_A \eta) = z_i (z^k (U_{p_i p_1}^{k_1} \cdots U_{p_i p_{i-1}}^{k_{i-1}} \eta)),
\]
which implies
\[
\pi_A V_{p_i} \pi_A^* = M_{z_i} (D_1 [U_{p_i p_1}] \cdots D_{i-1}[U_{p_i, p_{i-1}}]).
\]
Now suppose that $q_j \in A^c =\{q_1, \ldots, q_{n-m}\}$. Then \eqref{eqn: V_j V_i^m} and \eqref{eqn: Pi V_A} imply
\[
(\pi_A V_{q_j} \pi_A^*)(z^k \eta) = \pi_A V_{q_j} V^k_A \eta = \pi_A V^k_A (U^{k_1}_{q_j p_1} \cdots U^{k_m}_{q_j p_m} V_{q_j}|_{\cld_A} \eta) = z^k (U^{k_1}_{q_j p_1} \cdots U^{k_m}_{q_j p_m} V_{q_j}|_{\cld_A} \eta),
\]
as, by Lemma \ref{lemma: D_A reduces V_j}, $U^{k_1}_{q_j p_1} \cdots U^{k_m}_{q_j p_m} V_j|_{\cld_A} \eta = U^{k_1}_{q_j p_1} \cdots U^{k_m}_{q_j p_m} V_j \eta \in \cld_A$. Therefore
\[
\pi_A V_{q_j} \pi_A^* = (D_1[U_{q_j p_1}] \cdots D_m[U_{q_j p_m}]) (I_{H^2(\D^m)} \otimes V_{q_j}|_{\cld_A}) \qquad (j \in A^c).
\]
Finally, we consider the $n$-tuple $M_A = (M_{A,1}, \ldots, M_{A,n})$ on $H^2_{\cld_A}(\D^m)$ formed by the $m$ operators $\{\pi_A V_{p_i} \pi_A^*\}_{i=1}^m$ and $(n-m)$ operators $\{\pi_A V_{q_j} \pi_A^*\}_{j=1}^{n-m}$, where
\begin{equation}\label{eqn: tilde V_t}
M_{A,t}=
\begin{cases}
M_{z_1} & \mbox{if } t =p_1
\\
M_{z_i} (D_1 [U_{p_i p_1}] \cdots D_{i-1}[U_{p_i, p_{i-1}}]) & \mbox{if } t=p_i \mbox{ and }1<i\leq m
\\
(D_1[U_{q_j p_1}] \cdots D_m[U_{q_j p_m}]) (I_{H^2(\D^m)} \otimes V_{q_j}|_{\cld_A}) & \mbox{if } t = q_j \mbox{ and }1\leq j\leq n-m,
\end{cases}
\end{equation}
and $t \in \{1, \ldots, n\}$. Now the representation of the $A$-wandering data of $M_A$, denoted by $wd_{M_A}(A)$ (see Definition \ref{def: A wandering sub}), is essentially routine: Since $\ker M_{A,p_i}^* = \pi_A (\ker V_{p_i}^*)$ for all $i=1, \ldots, m$, it follows that $\cap_{p_i \in A} \ker M_{A,p_i}^* = \pi_A (\clw_A)$. For each $l \in \Z_+^{n-|A|}$, we have
\[
M_{I_n\setminus A}^l \Big(\bigcap_{p_i \in A} \ker M_{A,p_i}^*\Big) = (\pi_A {V}_{I_n \setminus A}^l \pi_A^*) (\pi_A \clw_A) = \pi_A {V}_{I_n \setminus A}^l \clw_A.
\]
This implies that $\cld_A(M_A) = \pi_A (\cld_A)$, and thus, by the definition $\pi_A$ (see \eqref{eqn: Pi V_A}), we get $\cld_A(M_A) = \cld_A$. Note that we are identifying $\cld_A$ with the set of all $\cld_A$-valued constant functions in $H^2_{\cld_A}(\D^m)$. Moreover, for each $q_j \in A^c$ and $f \in \cld_A$, since $\pi^*_A f = f$ and $V_{q_j} f \in \cld_A$, it follows that
\[
M_{A,q_j} f = \pi_A V_{q_j} \pi^*_A f = \pi_A V_{q_j} f = V_{q_j} f,
\]
and hence $M_{A, q_j}|_{\cld_A} = V_{q_j}|_{\cld_A}$. We summarize this observation as a proposition.

\begin{prop}\label{prop: model of V on H_A}
Let $(V_1, \dots, V_n)$ be a $\clu_n$-twisted isometry on $\clh$, and let $A \subseteq I_n$. If $\cld_A \neq \{0\}$, then the tuple $V|_{\clh_A}$ is unitarily equivalent to $M_A = (M_{A,1}, \ldots, M_{A,n})$ on $H^2_{\cld_A}(\D^{|A|})$, where $M_{A,i}$'s are defined as in \eqref{eqn: tilde V_t}. Moreover, if $A^c = \{q_1, \ldots, q_{n-m}\}$, then
\[
wd_{M_A}(A) = (I_{\cld_A}, V_{q_1}|_{\cld_A}, \dots, V_{q_{n-m}}|_{\cld_A}),
\]
and all other wandering data are zero tuples.
\end{prop}

We call $M_A$ the \textit{model operator corresponding to $A \subseteq I_n$} (or simply the \textit{model operator}). Note that the model operator $M_A$ on $H^2_{\cld_A}(\D^{|A|})$ is a $\clu_n$-twisted isometry, where $\clu_n = \{\pi_A U_{ij}\pi_A^*\}_{i \neq j}$.

In particular, if $A = \{1, \ldots, m\}$ for some $m \in \{1,\ldots, n\}$, then $V|_{\clh_A}$ on $\clh_A$ is unitarily equivalent to $M_A = (M_1, \ldots, M_n)$ on $H^2_{\cld_A}(\D^m)$, where $M_1 = M_{z_1}$ and
\[
M_i =  M_{z_i} (D_1 [U_{i 1}] D_2 [U_{i 2}] \cdots D_{i-1}[U_{i i-1}]),
\]
for all $i= 2,\ldots, m$, and
\[
M_j = (D_1[U_{j 1}] \cdots D_m[U_{j m}]) (I_{H^2(\D^m)} \otimes V_j|_{\cld_A}).
\]
for all $j= m+1, \ldots, n$, and $wd_{M_A}(A) = (I_{\cld_A}, V_{m+1}|_{\cld_A}, \dots, V_{n}|_{\cld_A})$.

Now we turn to analytic models of $\clu_n$-twisted isometries. Let $V = (V_1,\ldots, V_n)$ be an $\clu_n$-twisted isometry, and suppose $\clh = \oplus_{A \subseteq I_n} \clh_A$. To obtain the model of $V$, we will apply the above proposition for each $A \subseteq I_n$ and patch all the pieces together. Recall that, by convention, $H^2_{\cld_{\emptyset}}(\D^{|\emptyset|}) = \clh_{\emptyset}$, and $M_{\emptyset, t} = V_t|_{\clh_{\emptyset}}$ for all $t= 1, \ldots, n$. Proposition \ref{prop: model of V on H_A} now tells us that the $n$-tuples $V|_{\clh_A}$ and $M_A$ are unitarily equivalent via the canonical unitary $\pi_A : \clh_A \raro H^2_{\cld_A}(\D^{|A|})$ as defined in \eqref{eqn: Pi V_A}, where $\cld_A$ is the $A$-wandering subspace and $A$ is a non-empty subset of $I_n$.  Since $V_i=\oplus_{A \subseteq I_n}V_i|_{\clh_A}$ for all $i = 1, \ldots, n$, it follows that
\[
V = (V_1,\ldots,V_n)=\bigoplus\limits_{A \subseteq I_n}(V_1|_{\clh_A}, \ldots, V_n|_{\clh_A}).
\]
We set $M_{V,i} = \oplus_{A \subseteq I_n} M_{A,i} \in \clb(\oplus_{A \subseteq I_n} H^2_{\cld_A}(\D^{|A|}))$, $i=1, \ldots, n$, and define
\[
M_V = (M_{V,1}, \ldots, M_{V,n}).
\]
Then the unitary $\Pi_V:= \oplus_{A \subseteq I_n} \pi_A$ satisfies $\Pi_V V_i= M_{V,i} \Pi_V$ for all $i=1, \ldots, n$. Thus, we have proved:

\begin{thm}\label{thm: model}
Let $(V_1,\ldots,V_n)$ be a $\mathcal{U}_n$-twisted isometry on $\clh$. Then $(V_1,\ldots,V_n)$ is unitarily equivalent to $(M_{V,1}, \ldots, M_{V,n})$ on $\oplus_{A \subseteq I_n} H^2_{\cld_A}(\D^{|A|})$.
\end{thm}

In the case of doubly noncommuting isometries (that is, in the case $U_{ij} = z_{ij}I_{\clh}$), this was observed in \cite[Theorem 4.6]{JP}.

Note that the proof of the above theorem is a simple consequence of Proposition \ref{prop: model of V on H_A}, and the proof of Proposition \ref{prop: model of V on H_A} uses Theorem \ref{thm: main decomposition}. In the following, we present a second and somewhat more direct proof of Proposition \ref{prop: model of V on H_A}. The techniques of this proof may be of independent interest.

\noindent We begin with the case of a single isometry. Suppose $V \in \clb(\clh)$ is a shift, and suppose $\clw_V = \ker V^*$. Then we have the canonical unitary $\Pi_V: \clh \raro H^2(\D) \otimes \clw_V$ such that $\Pi_V V = (M_z \otimes I_{\clw_V}) \Pi_V$ (see the discussion preceding Definition \ref{def: A wandering sub}). Observe that
\begin{equation}\label{eqn: Pi^* V^j}
\Pi_V^* (z^j \otimes \eta) = V^j \eta \qquad (j \in \Z_+, \eta \in \clw_V).
\end{equation}
Now, let $1 \leq m \leq n$, and let $A = \{p_1, \ldots, p_m\} \subseteq I_n$. Let $V = (V_1, \ldots, V_n)$ be a $\clu_n$-twisted isometry. Suppose $V_i$ is a shift, and $V_j$ is a unitary for each $i \in A$ and $j \in A^c$, respectively. Set $\Pi_1:= \Pi_{V_{p_1}}$. By Lemma \ref{lemma: reducing}, we know that $\clw_{\{p_1\}}$ reduces $V_{p_2}$. Therefore, $V_{p_2}|_{\clw_{\{p_1\}}}$ is a shift in $\clb(\clw_{\{p_1\}})$. Lemma \ref{lemma: wandering} tells us that $\ker (V_{p_2}|_{\clw_{\{p_1\}}})^*= \clw_{\{p_1,p_2\}}$. Then the canonical unitary
\[
\Pi_2 := \Pi_{V_{p_2}|_{\clw_{\{p_1\}}}} : \clw_{\{p_1\}} \raro H^2(\D) \otimes \clw_{\{p_1, p_2\}},
\]
corresponding to $V_{p_2}|_{\clw_{\{p_1\}}}$ yields a unitary $I_{H^2(\D)} \otimes \Pi_2 : H^2_{\clw_{\{p_1\}}}(\D) \raro H^2_{\clw_{\{p_1, p_2\}}}(\D^2)$. Here we have once again used the identification $H^2_{\clw_{\{p_1, p_2\}}}(\D^2) = H^2(\D^2) \otimes {\clw_{\{p_1, p_2\}}}$. Continuing exactly in the same way, we find
\[
\begin{split}
0 \raro \clh \xrightarrow{\Pi_1} H^2_{\clw_{\{p_1\}}}(\D) \xrightarrow{I_{H^2(\D)} \otimes \Pi_2} H^2_{\clw_{\{p_1, p_2\}}}(\D^2) & \xrightarrow{I_{H^2(\D^2)} \otimes \Pi_3}
\\
& \qquad \cdots \xrightarrow{I_{H^2(\D^{m-1})} \otimes \Pi_m} H^2_{\clw_A}(\D^m) \raro 0.
\end{split}
\]
This gives us a unitary $\Pi : \clh \raro H^2_{\clw_A}(\D^m)$ defined by
\[
\Pi := (I_{H^2(\D^{m-1})} \otimes \Pi_m) (I_{H^2(\D^{m-2})} \otimes \Pi_{m-1}) \cdots (I_{H^2(\D)}  \otimes \Pi_2) \Pi_1.
\]
Now, for each $i=2, \ldots, m$, use \eqref{eqn: Pi^* V^j} to see that
\[
(I_{H^2(\D^{i-1})} \otimes \Pi_i)(z_1^{k_1} \cdots z_{i-1}^{k_{i-1}} \otimes V_{p_i}^{k_i}|_{\clw_{\{p_1,\ldots, p_{i-1}\}}} \eta) = z_1^{k_1} \cdots z_{i-1}^{k_{i-1}} z_i^{k_i} \eta,
\]
for all $k = (k_1, \ldots, k_{i-1}) \in \Z_+^{i-1}$, and $\eta \in \clw_{\{p_1,\ldots, p_{i-1}\}}$. Applying the above repeatedly, we find that $\Pi (V^k_A \eta) = z^k \eta$, $k \in \Z_+^m$, $\eta \in \clw_A$, which was obtained in \eqref{eqn: Pi V_A}. The remainder of the proof of Proposition \ref{prop: model of V on H_A} now proceeds similarly.

We should mention that the above techniques can be readily adapted to prove (at the expense of a more cumbersome computation) Theorem \ref{thm: model} in its full generality.

\section{Invariants}\label{sec: invariants}

The purpose of this section is to prove that wandering data are complete unitary invariants for $\clu_n$-twisted isometries. We start with a simple observation.

\begin{lem}\label{lemma: pi V  tilde V}
Let $V = (V_1, \ldots, V_n)$ on $\clh$ and $\tv = (\tv_1, \ldots, \tv_n)$ on $\tilde{\clh}$ be $\clu_n$-twisted isometries with respect to $\{U_{ij}\}_{i < j} \subseteq \clb(\clh)$ and $\{\tilde{U}_{ij}\}_{i < j} \subseteq \clb(\tilde{\clh})$, respectively, and let $\Pi: \clh \raro \tilde{\clh}$ be a unitary operator. If $\Pi V_i = \tv_i \Pi$ for all $i=1,\ldots, n$, then $\Pi U_{st} = \tilde{U}_{st} \Pi$ for all $s \neq t$.
\end{lem}
\begin{proof}
The proof follows at once from the fact that $U_{st} = V_s^* V_t^* V_s V_t$ for all $s \neq t$ (see \eqref{eqn: U_{ij} = V_i^* V_j^* V_i V_j}).
\end{proof}

In particular, if $V \cong \tv$, then the $\binom{n}{2}$-tuples $\{U_{ij}\}_{i < j}$ and $\{\tilde{U}_{ij}\}_{i < j}$ are unitarily equivalent under the same unitary map.

Let $V = (V_1, \ldots, V_n)$ on $\clh$ and $\tv = (\tv_1, \ldots, \tv_n)$ on $\tilde{\clh}$ be $\clu_n$-twisted isometries with respect to $\{U_{ij}\}_{i < j} \subseteq \clb(\clh)$ and $\{\tilde{U}_{ij}\}_{i < j} \subseteq \clb(\tilde{\clh})$, respectively. For $A \subseteq I_n$, we say that $wd_V(A)$ is \textit{twisted unitarily equivalent} to $wd_{\tv}(A)$ (which we will denote by $wd_V(A) \cong_{\clu} wd_{\tv}(A)$) if the tuples $wd_V(A) \cup \{U_{ij}|_{\cld_A(V)}\}_{i\neq j}$ and $wd_{\tv}(A)\cup \{\tilde U_{ij}|_{\cld_A(\tilde V)}\}_{i\neq j}$ are unitarily equivalent.

We are now all set to prove that $wd_V(A) \cup \{U_{ij}|_{\cld_A(V)}\}_{i\neq j}$ is a complete set of unitary invariants of $\clu_n$-twisted isometry $V$.

\begin{thm}\label{thm: unitary equivalent}
Let $V = (V_1, \ldots, V_n)$ on $\clh$ and $\tv = (\tv_1, \ldots, \tv_n)$ on $\tilde{\clh}$ be $\clu_n$-twisted isometries with respect to $\{U_{ij}\}_{i < j} \subseteq \clb(\clh)$ and $\{\tilde{U}_{ij}\}_{i < j} \subseteq \clb(\tilde{\clh})$, respectively. Then the following statements are equivalent:
\begin{enumerate}
\item $V \cong \tv$.
\item $wd_V(A) \cong_{\clu} wd_{\tv}(A)$ for all $A \subseteq I_n$.
\end{enumerate}
\end{thm}
\begin{proof}
$(1) \Rightarrow (2)$ Let $\pi:\clh \raro \tilde{\clh}$ be a unitary, and let $\pi V_i = \tv_i \pi$ for all $i=1,\ldots, n$. Fix $A \subseteq I_n$. Then (see the discussion preceding Proposition \ref{prop: model of V on H_A}) $\pi \clw_A = \tilde{\clw}_A$, where $\tilde{\clw}_A = \bigcap_{i\in A} \ker \tv_i^*$. Recall that $\cld_A(V)$ and $\cld_A(\tv)$ denotes the $A$-wandering subspaces of $V$ and $\tv$, respectively. Then
\[
\pi \cld_A(V) = \bigcap_{l \in \Z_+^{n-|A|}} \pi V_{I_n \setminus A}^l \pi^* (\pi \clw_A) = \bigcap_{l \in \Z_+^{n-|A|}} \tv_{I_n \setminus A}^l \tilde{\clw}_A = \cld_A(\tv),
\]
and hence, $\pi|_{\cld_A(V)} : \cld_A(V) \raro {\cld}_A(\tv)$ is a unitary. Now fix $j \in A^c$, $l \in \Z_+^{n - |A|}$, and $f \in \clw_A$. Since, by Lemma \ref{lemma: D_A reduces V_j}, $\cld_A(V)$ reduces $V_j$, it follows that
\[
(\pi|_{\cld_A(V)} V_j) V_{I_n \setminus A}^l f = (\tv_j \pi) V_{I_n \setminus A}^l f = (\tv_j \pi|_{\cld_A(V)}) V_{I_n \setminus A}^l f,
\]
as $V_{I_n \setminus A}^l f \in \cld_A(V)$. Therefore, $\pi|_{\cld_A(V)} V_j|_{\cld_A(V)} = \tilde V_j|_{\cld_A(\tilde{V})} \pi|_{\cld_A(V)}$ for all $j \in A^c$. Finally, $\pi|_{\cld_A(V)} U_{ij}|_{\cld_A(V)} = \tilde{U}_{ij}|_{\cld_A(\tilde{V})} \pi|_{\cld_A(V)}$ follows from the fact that $\pi U_{ij} = \tilde{U}_{ij} \pi$, $i\neq j$. This proves that $(1) \Rightarrow (2)$.

\noindent To prove $(2) \Rightarrow (1)$, we first consider orthogonal decompositions $\clh = \oplus_{A \subseteq I_n} \clh_A$ and $\tilde \clh = \oplus_{A \subseteq I_n} \tilde \clh_A$. Suppose $A = \{p_1,\ldots, p_m\} \subseteq I_n$. By assumption, there exists a unitary $\tau_A: \cld_A(V) \raro \cld_A(\tv)$ such that $\tau_A V_j|_{\cld_A(V)} = \tv_j|_{\cld_A(\tv)} \tau_A$ and $\tau_A U_{st}|_{\cld_A(V)} = \tilde{U}_{st}|_{\cld_A(\tv)} \tau_A$ for all $j \in A^c$ and $s \neq t$. We also know that $\clh_A = \oplus_{k \in \Z_+^{|A|}} V_A^k \cld_A(V)$ and $\tilde{\clh}_A = \oplus_{k \in \Z_+^{|A|}} \tv_A^k \cld_A(\tv)$ (see Theorem \ref{thm: main decomposition}). Then the map $\pi_A(V_A^k \eta) = \tv_A^k \tau_A \eta$, for all $k  \in \Z_+^{|A|}$ and $\eta \in \cld_A(V)$, defines a unitary $\pi_A : \clh_A \raro \tilde{\clh}_A$. Let $k \in \Z_+^{|A|}$ and $\eta \in \cld_A(V)$. For each $p_i \in A$, we have $(\pi_A V_{p_i}|_{\clh_A}) (V_A^k \eta) = \pi_A V_{p_i} V_A^k \eta$, and hence
\[
(\pi_A V_{p_i}|_{\clh_A}) (V_A^k \eta) = \pi_A (V_A^{k+e_i} (U_{p_i p_1}^{k_1} \cdots U_{p_i p_{i-1}}^{k_{i-1}} \eta)) =  \tilde{V}_A^{k+e_i} (\tau_A U_{p_i p_1}^{k_1} \cdots U_{p_i p_{i-1}}^{k_{i-1}}\eta).
\]
Since $\tau_A U_{p_i p_1}^{k_1} \cdots U_{p_i p_{i-1}}^{k_{i-1}}|_{\cld_A(V)} = \tilde{U}_{p_i p_1}^{k_1} \cdots \tilde{U}_{p_i p_{i-1}}^{k_{i-1}}|_{\cld_A(\tilde{V})}\tau_A$, we obtain
\[
(\pi_A V_{p_i}|_{\clh_A}) (V_A^k \eta) = (\tilde{V}_{p_i}|_{\tilde{\clh}_A} \pi_A) (V_A^k \eta),
\]
and hence $\pi_A V_{p_i}|_{\clh_A} = \tilde{V}_{p_i}|_{\tilde{\clh_A}} \pi_A$ for all $p_i \in A$. The remaining equality $\pi_A V_i|_{\clh_A} = \tilde{V}_i|_{\tilde{\clh}_A} \pi_A$ for all $i \in A^c$ is similar. Now we consider the unitary $\pi := \oplus_{A \subseteq I_n} \pi_A: \oplus_{A \subseteq I_n} \clh_A = \clh \longrightarrow \oplus_{A \subseteq I_n} \tilde{\clh}_A = \tilde{\clh}$. Since $V_j = \oplus_{A \subseteq I_n} V_j|_{\clh_A}$ and $\tv_j = \oplus_{A \subseteq I_n} \tv_j|_{\tilde{\clh}_A}$, by the previous identity, we have $\pi V_j = \tv_j \pi$ for all $j \in I_n$. Finally, since $U_{ij} = V_i^*V_j^*V_iV_j$ and $\tilde{U}_{ij} = \tv_i^* \tv_j^* \tv_i \tv_j$, it follows that
\[
\pi U_{ij} = (\oplus_{A \subseteq I_n} \pi_A) (\oplus_{A \subseteq I_n} U_{ij}|_{\clh_A}) = (\oplus_{A \subseteq I_n} \tilde{U}_{ij}|_{\tilde{\clh}_A}) (\oplus_{A \subseteq I_n} \pi_A) = \tilde{U}_{ij} \pi,
\]
and completes the proof of the theorem.
\end{proof}

\section{Nuclear $C^*$-algebras}\label{sec: nuclear}

Our objective in this section is to show that the universal $C^*$-algebra generated by a $\clu_n$-twisted isometry, $n \geq 2$, is nuclear.

We begin by recalling the definition of a universal $C^*$-algebra (cf. \cite[page 885]{W1}). Let $\clg = \{g_i : i\in \Lambda\}$ be a set of generators and $\clr$ be a set of relations. A unital $C^*$-algebra $A$ is said to be a \textit{universal $C^*$-algebra} generated by the elements in $\clg$ and
satisfying the relation $\clr$ if it satisfies the following property: If $\tilde{A}$ is a unital $C^*$-algebra generated by $\tilde{\clg} = \{\tilde{g}_i: i \in \Lambda\}$ that satisfies the same relation set $\clr$, then there exists a unique $*$-epimorphism $\pi: A \raro \tilde{A}$ such that $\pi(g_i) = \tilde{g}_i$ for all $i\in \Lambda$.

Given $C^*$-algebras $A$ and $B$, we denote by $A\otimes B$ the algebraic tensor product of $A$ and $B$. A norm $\|\cdot\|_{\alpha}$ on $A \otimes B$ is said to be a $C^*$-norm if $\norm{xy}_{\alpha} \leq \norm{x}_{\alpha} \norm{y}_{\alpha}$ and $\norm{x^*x}_{\alpha}=\norm{x}_{\alpha}^2$ holds for all $x$ and $y$ in $A\otimes B$.

The minimal tensor norm $\|\cdot\|_{min}$ and the maximal tensor norm $\|\cdot\|_{max}$ are the extreme examples of $C^*$-norms: If $\|\cdot\|_{\alpha}$ is a $C^*$-norm on the algebraic tensor product $A\otimes B$, then
\[
\|x\|_{min}\leq \|x\|_{\alpha} \leq \|x\|_{max} \qquad (x\in A \otimes B).
\]
Finally, we recall that a ${C}^*$-algebra $A$ is called \textit{nuclear} \cite[page 184]{Blackadar} if for each ${C}^*$-algebra $B$ there is a unique ${C}^*$-norm on $A\otimes B$. It is well known that a ${C}^*$-algebra $A$ is nuclear if and only if $\|x\|_{min}=\|x\|_{max}$ for all $x\in A\otimes B$ and all $C^*$-algebras $B$.

We now return to $\clu_n$-twisted isometries. We denote by $\mathcal{T}_n$, the universal $\mathrm{C}^*$-algebra generated by the set
\[
\clg = \{X_t, X_{ij}: i, j, t \in I_n, i< j\},
\]
with the set of relations $\clr = \clr_1 \cup \clr_2$, where
\[
\clr_1=\{r^\ast r = I, X_tX_{ij} = X_{ij}X_t : r = X_t, X_{ij}, X_{ij}^\ast, i < j\},
\]
and
\[
\clr_2 = \{X_{ji}=X_{ij}^\ast, X_i^\ast X_j=X_{ij}^\ast X_jX_i^\ast : i < j\}.
\]
Evidently, $\mathcal{T}_n$ is generated by $n$ isometries $(V_1, \ldots, V_n)$ and $\binom{n}{2}$ unitaries $\{U_{ij}\}_{i<j}$ satisfying \eqref{equation: u_ntwisted} with $U_{ji}=U_{ij}^*$ for all $i<j$.

We wish to point out that Proskurin \cite{P} and Weber \cite{W} proved that the universal ${C}^*$-algebra generated by a doubly non-commuting pair of isometries (that is, in the case of $U_{ij} = z_{ij} I_{\clh}$, $i \neq j$) is nuclear (also see \cite{Kuzmin 1, Kuzmin 2, PV} for other relevant results). The main tool used in \cite{P,W} is a result of Rosenberg \cite[Theorem 3]{R}, which determines amenability of $C^*$-algebras generated by amenable $C^*$-subalgebras (recall that all nuclear $C^*$-algebras are amenable \cite{UH}):

\begin{thm}[Rosenberg]\label{thm: Rosenberg}
Let $A$ be a unital ${C}^*$-algebra generated by a nuclear  ${C}^*$-subalgebra $B$ containing the unit of $A$ and an isometry $s\in A$ satisfying the condition $sBs^*\subseteq B$. Then $A$ is nuclear.
\end{thm}

We are now ready to prove that $\mathcal{T}_n$ is nuclear. Here also, the above criterion will play a key role.

\begin{thm}\label{nuclearity}
$\mathcal{T}_n$ is nuclear for $n\geq 2$.	
\end{thm}
\begin{proof}
Note that $\clt_n$ is generated by isometries $\{V_i\}_{i=1}^n$ and unitaries $\{U_{ij}\}_{i \neq j}$, where $V=(V_1, \ldots, V_n)$ is a $\clu_n$-twisted isometry with respect to $\{U_{ij}\}_{i < j}$. For each $m \in \Z_+$, we set
\[
P_i(m) = V_i^{m}V_i^{*m} \qquad (i=1,\ldots, n).
\]
Clearly, $\{P_1(m_1), \ldots, P_n(m_n)\}$ are orthogonal projections for all $m_i \in \Z_+$, $i=1, \ldots, n$. Consider the $C^*$-algebra
\[
\clb= C^*(\{P_1(k_1),P_2(k_2),\cdots ,P_n(k_n), U_{jk}: k_1, \ldots,k_n \in \Z_+ ~\mbox{and } 1\leq j < k \leq n\}).
\]
Now $\{U_{jk}\}_{j \neq k}$ is a commuting family (see Remark \ref{remark: Uij = V*V*VV}). Since $V_i \in \{U_{jk}\}'_{j \neq k}$, it follows that $P_i(k_i) \in \{U_{jk}\}'_{j \neq k}$  for all $i=1, \ldots, n$, and $k_i \in \Z_+$. Also  $$P_i(k_i)P_j(k_j)=P_j(k_j)P_i(k_i), ~~\mbox{for all } k_i, k_j \in \Z_+ ~\mbox{and } 1 \leq i,j \leq n.$$ implies that the elements $P_1(k_1)P_2(k_2)\cdots P_n(k_n)$ commutes among themselves for all $k_i\in \Z_+$ with $1\leq i\leq n$. Therefore $\clb$ is a commutative $C^*$-algebra. Since commutative $C^*$-algebras are nuclear, it follows, in particular, that $\clb$ is nuclear. Finally, since
\[
\clt_n=C^*(\{\clb, V_i : 1\leq i \leq n\}),
\]
a repeated application ($n$-times) of Theorem \ref{thm: Rosenberg} to $\clb$ and $V_i$'s proves that $\clt_n$ is nuclear.
\end{proof}

The following observation, in particular, also says that the $C^*$-algebra generated by a $\clu_2$-twisted isometry is not simple (see \cite[Section II.5.4]{Blackadar} on simple $C^*$-algebras).

\begin{remark}\label{lemma: ideal of compact}
Let $\clk$ be the universal $C^*$-algebra of compact operators on a separable infinite dimensional Hilbert space generated by self-adjoint matrix units $\{E_{ij}\}_{i,j \in \Z_+}$ (that is, $E_{ij}E_{kl}=\delta_{jk}E_{il}$ and $E_{ij}^*=E_{ji}$ for all $i,j,k,l \in \Z_+$). Let $U$ be a unitary and $(V_1,V_2)$ be a pair of isometries acting on a Hilbert space $\clh$. Suppose $V_1^*V_2 = U^* V_2V_1^*$, and denote by $\langle (1- V_1V_1^*)(1-V_2V_2^*)\rangle$ the ideal generated by $(1- V_1V_1^*)(1-V_2V_2^*)$ in $C^*(V_1,V_2)$. For $p, q, r, s \in \Z_+$, define
\[
E_{pq, sr}:= V_1^p V_2^q (1-V_1 V_1^*)(1-V_2 V_2^*)V_2^{*r} V_1^{*s}.
\]
It is easy to check that,
\[
E^*_{pq, sr}= E_{sr,pq} \quad \text{and} \quad E_{pq,sr} E_{ij,lk} = \delta_{si} \delta_{rj} E_{pq,lk}.
\]
for all $a,b,c,d, i, j, k, l \in \Z_+$, that is, $\{E_{pq,sr}\}_{p,q,r,s \in \Z_+}$ is a self-adjoint system of matrix units. Using the universal property of $\clk$ and the fact that $\clk$ is simple, we conclude that $\clk$ is isomorphic to the closed subalgebra of $\langle (1-V_1V_1^*)(1-V_2V_2^*)\rangle$ spanned by $\{E_{pq,sr}\}_{p,q,r,s \in \Z_+}$. Therefore the proper ideal $\langle (1- V_1V_1^*)(1-V_2 V_2^*)\rangle$ in $C^*(V_1,V_2)$ contains a subalgebra isomorphic to $\clk$.
\end{remark}

\section{Classifications}\label{sec: Classification}

In this section, we classify $\clu_n$-twisted isometries via representations of twisted noncommutative tori.
	
We begin by recalling the definitions of rotation algebras or noncommutative tori and the Heisenberg group $C^*$-algebras (see \cite{AP, Packer, PR} for more details). For $\theta \in \R$, the \textit{rotation algebra} is defined as the universal $C^*$-algebra
\[
\mathcal{A}_{\theta}:=C^*(\{U_1, U_2: U_1, U_2 \text{ are unitaries and } U_1 U_2=e^{2\pi i\theta}U_2 U_1 \}).
\]
The rotation algebra is also known as the \textit{noncommutative torus} as for $\theta=0$, $\mathcal{A}_0\cong C(\T^2)$, where $\T$ denotes the unit circle. When $\theta$ is irrational, $\mathcal{A}_{\theta}$ is called the irrational rotation algebra which is a simple $C^*$-algebra having the unique faithful trace $\tau_{\theta}:\mathcal{A}_{\theta}\to \C$ defined by
\[
\tau_{\theta}(U_1^l U_2^m)=
\begin{cases}
1 & \quad\text{if } l=m=0\\ 0 &\quad \text{otherwise },
\end{cases}
\]
for $l,m\in \Z$. Let $\mathcal{A}=C^*(H)$, the group $C^*$-algebra of the Heisenberg group
\[
H:= \left\lbrace
\begin{pmatrix}
1 & m & p\\
0 & 1 & n \\
0 & 0 & 1
\end{pmatrix}: m,n,p\in \Z
\right\rbrace.
\]
We can view $\mathcal{A}$ as the universal $C^*$-algebra generated by three unitaries $u,v,w$ satisfying
\[
u,v\in \{ w\}' ~~ \text{and}~~ uv=wvu.
\]
It is known \cite{PR} that $\mathcal{A}$ has a central-valued trace $\tau :\mathcal{A}\to C^*(w)$ defined by
\[
\tau(w^ku^lv^m):=\begin{cases}
w^k & \text{if} ~~l=m=0\\
0 & \text{otherwise},
\end{cases}
\]
for $k,l,m\in \Z$ where $C^*(w)$ is the center of $\mathcal{A}$.

With this motivational background, we finally recall the notion of generalized Heisenberg group (see \cite{Tron} for more details). For each $n\geq 2$, the \textit{generalized Heisenberg group} $G(n)$ is the group generated by unitaries $\{U_i:1\leq i\leq n\}$ and $\{U_{jk}: 1\leq j< k\leq n\}$ satisfying the relations
\begin{equation}\label{equation: generalized Heisenberg group $C^*$-algebra}
U_i U_{jk} = U_{jk} U_i \text{ and } U_j U_k = U_{jk} U_k U_j \;\;\mbox{for all }1\leq i\leq n, \mbox{ and } 1\leq j<k\leq n.
	\end{equation}
The group $C^*$-algebra $C^*(G(n))$ associated with $G(n)$ is the universal $C^*$-algebra generated by unitaries $\{U_i:1\leq i\leq n\} \cup \{U_{jk}: 1\leq j< k\leq n\}$ satisfying \eqref{equation: generalized Heisenberg group $C^*$-algebra}.
	
Recall that a \emph{representation} of a $C^*$-algebra $\mathcal{A}$ is a pair $(\pi ,\clh)$, where $\clh$ is a Hilbert space and $\pi :\mathcal{A} \to \clb(\clh)$ is a $*$-homomorphism. If $\mathcal{A}$ is unital, then $\pi$ is assumed to be unital. Any pair $(\clu,\clh)$, where $\clu = \{U_i, U_{jk} : 1 \leq i, j, k \leq n, \mbox{ and } j\neq k\} \subseteq \clb(\clh)$ is a collection of unitaries with $U_{kj}= U_{jk}^*$ for all $1\leq j<k\leq n$ satisfying \eqref{equation: generalized Heisenberg group $C^*$-algebra} is called a \emph{representation} of $C^*(G(n))$. Two representations $(\clu, \mathcal{H})$ and $(\clw, \mathcal{K})$ are said to be \textit{unitarily equivalent} if there is a unitary $\eta :\mathcal{H} \to \mathcal{K}$ such that $\eta U_i = W_i \eta$ for all $1\leq i\leq n$. In this case, by Lemma \ref{lemma: pi V  tilde V}, it also follows that $\eta U_{ij} = W_{ij} \eta$ for all $i \neq j$.
	
Let $(\clu,\clh)$ be a representation of $C^*(G(n))$, and let
\[
\clw = \{W_i, W_{jk}: 1\leq i,j,k\leq n \mbox{ and } j< k \} \subseteq \clb(\clk),
\]
be a generating set of $C^*(G(n))$, where $W_{kj}= W_{jk}^*$ for all $1\leq j<k\leq n$. Then from the universal property of $C^*(G(n))$, it follows that there exists a unique representation $\pi :C^*(G(n)) \to \clb(\clk)$ such that $\pi(U_i)= W_i$ and $\pi(U_{jk}) = W_{jk}$ for all $ i,j,k$, and $ j\neq k$.

Now we return to $\clu_n$-twisted isometries. Let $U=(U_1,\ldots ,U_n)$ be a $\clu _n$-twisted isometry with respect to $\{U_{jk}\}_{j< k} \subseteq \clb(\clh)$, and let $A\subseteq I_n$. Set $\tilde{U_i}:=U_i|_{\clh_A}, \tilde{U}_{jk}:=U_{jk}|_{\clh_A}$ for all $i \in A^c$ and $j \neq k$. Then $\tilde{U_i}$ and $\tilde{U}_{jk}$ are unitaries satisfying
\begin{equation}\label{equation: T_A}
\tilde{U}_{kj}=\tilde{U}_{jk}^*, ~~
\tilde{U_i}\tilde{U}_{jk}=\tilde{U}_{jk}\tilde{U_i}, ~~\tilde{U}_{i_1}\tilde{U}_{i_2}= \tilde{U}_{i_1i_2}\tilde{U}_{i_2}\tilde{U}_{i_1},\mbox{ and } \tilde{U}_{jk}\tilde{U}_{lm}=\tilde{U}_{lm}\tilde{U}_{jk} ~~
\end{equation}
for all $i,{i_1},{i_2}\in A^c$, $1\leq j\neq k\leq n$, and $1\leq l\neq m\leq n$. We denote by $\T_A$ the universal $C^*$-algebra generated by $\{\tilde{U_i},\tilde{U}_{jk}: i\in A^c, j\neq k\}$ satisfying \eqref{equation: T_A} and call it the \textit{twisted noncommutative torus with respect to $A$} (or simply \textit{twisted noncommutative torus} if $A$ is clear from the context). Let $A \subseteq I_n$ with $|A|=m$ and let $(\mathcal{V}, \clw)$ be a representation of $\T_A$, where $\mathcal{V} = \{V_i, V_{jk}: i\in A^c, 1\leq j\neq k \leq n\} \subseteq \clb(\clw)$. Then there exists a $\clu_n$-twisted isometry $M_A$ such that
\begin{equation}\label{eqn: wd_V(B) =S}
wd_{M_A}(B) =
\begin{cases}
\{I_{\clw}, V_i: i \in A^c\} & \mbox{if } B = A
\\
\{0\} & \mbox{if } B \neq A.
\end{cases}
\end{equation}
Indeed, following the construction and properties of the model operators in Sections \ref{sec: example} and \ref{sec: model and wand data}, we set $U_{ij} = I_{H^2(\D^m)}\otimes V_{ij}$ for all $i \neq j$. Clearly, $U_{ij}^*=U_{ji}$ for all $i \neq j$. Then $\{U_{ij}\}_{i \neq j} \subseteq \clb(H^2_{\clw}(\D^m))$, and $M_A := (M_{A,1},\ldots, M_{A, n})$ on $H^2_{\clw}(\D^m)$ is a $\clu_n$-twisted isometry with respect to $\{U_{ij}\}_{i<j}$, where $M_{A, i}$'s are defined as in \eqref{eqn: tilde V_t}. Moreover, by Proposition \ref{prop: model of V on H_A} we obtain the desired equality \eqref{eqn: wd_V(B) =S}.
	
Also note that two representations of $\mathbb{T}_A$ are unitarily equivalent if and only if the corresponding $A$-wandering data are twisted unitarily equivalent. This leads to the following generalization of \cite[Theorem 5.3]{JP}:

\begin{thm}\label{thm: classification}
The unitary equivalence classes of $\clu_n$-twisted isometries are in bijection with enumerations of $2^n$ unitary equivalence classes of unital representations of twisted noncommutative tori $\mathbb{T}_A$, $A\subseteq I_n$.
\end{thm}
\begin{proof}
Suppose $V:=(V_1,\dots ,V_n)$ is a $\clu_n$-twisted isometry on a Hilbert space $\mathcal{H}$. Then for each $A\subseteq I_n$, the pair $(\pi_V(A), \cld_A)$ is a representation of $\T_{A}$, where
\[
\pi_V(A) := \{V_i|_{\cld_A}, U_{jk}|_{\cld_A}: i \in A^c, j \neq k\}.
\]
Well-definedness and injectivity of the correspondence
\[
V \leftrightarrow \{(\pi_V(A), \cld_A): A \subseteq I_n\},
\]
follow from Theorem \ref{thm: unitary equivalent}. Now we check the surjectivity of this correspondence. Consider $\{(\clr_{\T_A}, \clw_A)\}_{A \subseteq I_n} $, where $(\clr_{\T_A}, \clw_A)$ is a representation of $\T_A$ and
\[
\clr_{\T_A}=\{V_{A,i}, V_{A,jk}: i \in A^c, ~j \neq k \} \subseteq \clb(\clw_A),
\]
for all $A \subseteq I_n$. Our aim is to construct a $\clu_n$-twisted isometry $M=(M_1,\ldots,M_n)$ on some Hilbert space $\clh$ such that $\pi_M(A)=\clr_{\T_A}$ for all $A \subseteq I_n$. Indeed, following the construction preceding the statement of this theorem, for each $A \subseteq I_n$, there exists a $\clu_n$-twisted isometry $M_A = (M_{A,1},\ldots, M_{A, n})$ with respect to $\{U_{A, ij}\}_{i < j}$ on $H^2_{\clw_A}(\D^{|A|})$ such that
\[
\cld_A(M_B)=\begin{cases}
\clw_A & \mbox{if } B = A
\\
\{0\} & \mbox{if } B \neq A,
\end{cases}
\]
and
\begin{equation}\label{eqn: pi_{M_B}(A)}
\pi_{M_B}(A) =
\begin{cases}
\clr_{\T_A} & \mbox{if } B = A
\\
\{0\} & \mbox{if } B \neq A.
\end{cases}
\end{equation}
Define $\clh = \oplus_{A \subseteq I_n} H^2_{\clw_A}(\D^{|A|})$ and
\[
M=\bigoplus\limits_{A \subseteq I_n}M_A=(\bigoplus\limits_{A \subseteq I_n}M_{A,1},\ldots, \bigoplus\limits_{A \subseteq I_n}M_{A,n}).
\]
Clearly, $M$ is a $\clu_n$-twisted isometry with respect to $\{U_{ij}\}_{i < j}:=\{\oplus_{A \subseteq I_n}U_{A,ij}\}_{i < j}$. It remains to check that $\pi_M(A)=\clr_{\T_A}$ for all $A$.
		
\noindent Fix $A \subseteq I_n$, and suppose $A =\{p_1,\ldots,p_m\} \subseteq I_n$ and $A^c=\{q_1,\ldots, q_{n-m}\}$. As
\[
\cld_A(M)=\bigcap\limits_{k_1,\ldots,k_{n-m}\in \Z_+}M_{q_1}^{k_1}\cdots M_{q_{n-m}}^{k_{n-m}}(\bigcap_{i=1}^{m}\ker M_{p_i}^*)=\bigoplus\limits_{B \subseteq I_n}\cld_A(M_B) \cong \clw_A,
\]
so (\ref{eqn: pi_{M_B}(A)}) gives us
\[
\pi_M(A)=\{M_i|_{\cld_A(M)}, U_{jk}|_{\cld_A(M)}: i \in A^c, j \neq k\}=\bigoplus\limits_{B \subseteq I_n}\pi_{M_B}(A) \cong \pi_{M_A}(A)=\clr_{\T_A},
\]
which completes the proof.
\end{proof}

Before proceeding we need to clarify the issue of reducing subspaces of model operators. First, given an $m$-tuple $X = (X_1, \ldots, X_m)$ on a Hilbert space $\clh$, we define the \textit{defect operator} $\mathbb{S}_m^{-1}(X, X^*)$ by
\[
\mathbb{S}_m^{-1}(X, X^*) = 1+\sum_{1 \leq i_1 < \ldots <i_t \leq m} (-1)^{t} X_{i_1} \cdots X_{i_t} X_{i_t}^* \cdots X_{i_1}^*.
\]
It should be noted that the above (well known) notion is inspired by the so-called hereditary functional calculus corresponding to the polynomial
\[
\mathbb{S}_m^{-1}(z,w) = 1+\sum_{1 \leq i_1 < \ldots <i_t \leq m} (-1)^{t} z_{i_1} \cdots z_{i_t} \bar{w}_{i_t} \cdots \bar{w}_{i_1},
\]
where $\mathbb{S}_m(z, w) = \prod_{i=1}^m (1 - z_i \bar{w}_i)^{-1}$, $z, w \in \D^m$, is the Szeg\"{o} kernel of the polydisc $\D^m$. In fact, if we consider $M_z :=(M_{z_1}, \ldots, M_{z_m})$ on $H^2_{\cle}(\D^m)$ for some Hilbert space $\cle$, then an easy computation (for instance, action of $\mathbb{S}_m^{-1}(M_z, M_z^*)$ on monomials) reveals that $\mathbb{S}_m^{-1}(M_z, M_z^*) = P_{\mathbb{C}} \otimes I_{\cle}$, where $P_{\mathbb{C}}$ denote the orthogonal projection of $H^2(\D^m)$ onto the space of all constant functions. Now, let $(V_1, \ldots, V_n)$ be a $\clu_n$-twisted isometry, and let $A = \{p_1, \ldots, p_m\} \subseteq I_n$. Consider the model operator $M_A = (M_{A,1},\ldots, M_{A,n})$ on $H^2_{\cld_A}(\D^{|A|})$ (see Proposition \ref{prop: model of V on H_A}). By Lemma \ref{lemma:key index}, we have
\[
M_{A,p_i} M_{z_j} = M_{z_j} M_{A,p_i} \qquad (p_i < j).
\]
Let us denote $M_{A,z} = (M_{A,p_1},\ldots, M_{A,p_m})$ for simplicity. For each $p_i \in A$, Lemma \ref{lemma:key index} again implies that $M_{A,p_i} M_{A,p_i}^* = M_{z_i} M_{z_i}^*$. Then the preceding equality yields
\[
\mathbb{S}_m^{-1}(M_{A,z}, M_{A,z}^*) = \mathbb{S}_m^{-1}(M_z, M_z^*) = P_{\mathbb{C}} \otimes I_{\cld_A}.
\]
Now assume that $\cls \subseteq  H^2_{\cld_A}(\D^m)$ is a closed subspace, and suppose that $\cls$ reduces $M_A$. In particular, $\cls$ reduces $M_{A,z}$, and hence by the previous identity it follows that $f(0) = (P_{\mathbb{C}} \otimes I_{\cld_A})f \in \cls$ for all $f \in \cls$. Therefore, $\cls = H^2_{\cld}(\D^m)$, where $\cld = \overline{\mbox{span}} \{f(0): f \in \cls\}$ is a closed subspace of $\cld_A$. Finally, by the representation of $M_{A,q_j}$ in \eqref{eqn: tilde V_t}, we have that $\cld$ reduces $V_{q_j}|_{\cld_A}$ and $U_{st}|_{\cld_A}$ for all $q_j \in A^c$ and $1\leq s \neq t \leq n$ respectively. We summarize this (along with the trivial converse) as follows:
	
\begin{prop}\label{prop: reducing subspace}
Let $(V_1, \ldots, V_n)$ be a $\clu_n$-twisted isometry, and let $M_A$ be the model operator corresponding to $A \subseteq I_n$. Suppose $\cls \subseteq H^2_{\cld_A}(\D^{|A|})$ is a closed subspace. Then $\cls$ reduces $M_A$ if and only if there exists a closed subspace $\cld \subseteq \cld_A$ such that $\cld$ reduces $V_j|_{\cld_A}$ and $U_{st}|_{\cld_A}$ for all $j \in A^c$ and $s \neq t$, and $\cls = H^2_{\cld}(\D^m)$.
\end{prop}
	
Given a $\clu_n$-twisted isometry $V = (V_1, \ldots, V_n)$, we denote by $C^*(V)$, the $C^*$-algebra generated by $\{V_i\}_{i=1}^{n}$. Evidently, $C^*(V)$ is unital. A subspace $\cld \subseteq \clh$ is said to be \emph{invariant} under $C^*(V)$ if $T \cld \subseteq \cld$ for all $T \in C^*(V)$. It is easy to check that $\cld$ is invariant under $C^*(V)$ if and only if $\cld$ reduces $T$ for all $T \in C^*(V)$ or, equivalently, $\cld$ reduces $V_i$ for all $i \in I_n$. We refer the reader to \eqref{eqn: Pi V_A} and Proposition \ref{prop: model of V on H_A} to recall the definitions of the canonical unitary $\pi_A$ and the model operator tuple $M_A$, respectively. Also recall the definition of $\pi_V(A)$ from the proof of Theorem \ref{thm: classification}. The next theorem is analogous to \cite[ Theorem 5.4]{JP}.
	
\begin{thm}\label{thm: trivial invariant C*}
Let $V = (V_1,\ldots,V_n)$ be a $\clu_n$-twisted isometry on $\clh$ with $\clu_n=\{U_{jk}\}_{j \neq k}$. The following are equivalent.
\begin{enumerate}
\item  Only trivial subspaces of $\clh$ are closed and invariant under $C^*(V)$.
			
\item There exists $A\subseteq I_n$ such that $V \cong M_A$ and $\cld_A(M_A)$ has only trivial subspaces that are invariant under $C^*(\pi_{M_A}(A))$.
\end{enumerate}
\end{thm}
\begin{proof}
$(1)\Rightarrow(2)$: Evidently, $\clh = \clh_A$ for some $A \subseteq I_n$, and hence $V \cong M_A$, where $M_A$ is a $\clu_n$-twisted isometry on $H^2_{\cld_A}(\D^{|A|})$. So the only trivial subspaces of $H^2_{\cld_A}(\D^{|A|})$ are closed and invariant under $C^*(M_A)$. The rest follows from Proposition \ref{prop: reducing subspace}.

\noindent Similarly, $(2)\Rightarrow(1)$ follows from Proposition \ref{prop: reducing subspace}.
\end{proof}

\begin{cor}
Let $(V_1,\ldots,V_n)$ be a $\clu_n$-twisted isometry and $A \subseteq I_n$ such that $V_i$ are shifts for $i \in A$ and are unitaries for $i \in A^c$ with
\[
\dim (\bigcap\limits_{i \in A} \ker V_i^*)=1,
\]
then $C^*(V_1,\ldots,V_n)$ is irreducible. In particular, if $(V_1,\ldots,V_n)$ are $\clu_n$-twisted shifts with $\dim(\bigcap\limits_{i \in I_n}\ker V_i^*)=1$, then $C^*(V_1,\ldots,V_n)$ is irreducible.
\end{cor}
\begin{example}
The multiplication operators $(M_{z_1}, \ldots, M_{z_n})$ by the co-ordinate functions on the Hardy space $H^2(\D^n)$ with $n \geq 2$ generate an irreducible $C^*$-algebra.
\end{example}

Recall again, given a \textit{representation} $(\mathcal{H},\pi)$ of a unital $C^*$-algebra $\mathcal{A}$, a closed subspace $\cld \subseteq \clh$ \textit{reduces} $\pi$ if $\cld$ reduces $\pi(a)$ for all $a \in A$. A representation $(\mathcal{H},\pi)$ is called \textit{irreducible} if trivial subspaces are the only reducing subspaces of $\pi$. Clearly, if $\{s_i:i\in I\}$ is a generating set of a $C^*$-algebra $\mathcal{A}$, then a closed subspace $\mathcal{D} \subseteq \clh$ reduces $\pi$ if and only if it reduces $\pi(s_i)$ for all $i\in I$. The following which is an analogous version of \cite[Corollary 5.5]{JP}, is now an immediate consequence of Theorems \ref{thm: classification} and \ref{thm: trivial invariant C*}.

\begin{cor}\label{cor: classification of noncommutative tori}
The unitary equivalence classes of the nonzero irreducible representations of the $C^*$-algebras generated by $\clu_n$-twisted isometries are parameterized by the unitary equivalence classes of the non-zero irreducible representations of $2^n$ noncommutative tori $\T_A$, with $A \subseteq I_n$.
\end{cor}

We finally remark that the examples in Section \ref{sec: example} are the basic building blocks of $\clu_n$-twisted isometries. The same construction can also be applied to produce more natural examples of tuples of operators (for instance, replace the unitary $U$ in $D[U]$ by some isometry $V$). The present findings also suggest the following natural question: Classify $C^*$-algebras generated by tuples of isometries $(V_1, \ldots, V_n)$ on $\clh$ that satisfies $V_i V_j = U_{ij} V_j V_i$, where $\{U_{ij}\}_{i\neq j} \subseteq \clb(\clh)$ are unitaries. We hope in the near future to be able to present results in some of these natural directions.

\smallskip

\noindent\textbf{Acknowledgement:}
We are grateful to the referee for a careful reading, helpful suggestions, and for pointing out recent developments and relevant connections. The first author thanks to NBHM postdoctoral fellowship (India) for financial support.
The second named author is supported in part by the Mathematical Research Impact Centric Support, MATRICS (MTR/2017/000522), and Core Research Grant (CRG/2019/000908), by SERB (DST), and NBHM (NBHM/R.P.64/2014), Government of India.

\end{document}